\documentclass{amsart}

\usepackage{frcursive}
\usepackage{amsmath}
\usepackage{amsthm}
\usepackage{amssymb}
\usepackage{hyperref}
\usepackage{color}
\usepackage{mathrsfs}

\def\Ee{{\mathscr{E}}}

\def\AA{{\mathcal A}}
\def\BB{{\mathcal B}}

\def\DD{{\mathcal D}}

\def\HH{{\mathcal H}}
\def\II{{\mathcal I}}
\def\JJ{{\mathcal J}}

\def\LL{{\mathcal L}}
\def\MM{{\mathcal M}}

\def\QQ{{\mathcal Q}}

\def\WW{{\mathcal W}}
\def\XX{{\mathcal X}}
\def\YY{{\mathcal Y}}
\def\ZZ{{\mathcal Z}}

\def\LLL{{\mathscr{L}}}

\newtheorem{thm}{Theorem}[section]
\newtheorem{lem}[thm]{Lemma}

\newtheorem{cor}[thm]{Corollary}
\newtheorem{defi}[thm]{Definition}
\newtheorem{prop}[thm]{Proposition}
\newtheorem{rk}[thm]{Remark}

\newenvironment{preuve}{\vip\noindent {\it Proof}}{\hfill$\square$\vip}

\newcommand{\beqn}{\begin{equation}}
\newcommand{\eeqn}{\end{equation}}
\newcommand{\bear}{\begin{eqnarray}}
\newcommand{\eear}{\end{eqnarray}}
\newcommand{\bean}{\begin{eqnarray*}}
\newcommand{\eean}{\end{eqnarray*}}

\newcommand{\wto}{\, \rightharpoonup \,}
\newcommand{\rr}{{\mathbb{R}}}
\newcommand{\R}{{\mathbb{R}}}
\newcommand{\T}{{\mathbb{T}}}

\newcommand{\e}{\varepsilon}
\newcommand{\eps}{\varepsilon}
\newcommand{\vip}{\vskip.2cm}
\newcommand{\indiq}{\hbox{\rm 1}{\hskip -2.8 pt}\hbox{\rm I}}
\newcommand{\E}{\mathbb{E}}
\newcommand{\Prob}{\mathbb{P}}
\newcommand{\intot}{\int _0^t }

\newcommand{\intrd}{\int_{\rr^2}}

\newcommand{\cS}{{\mathcal S}}
\newcommand{\cK}{{\mathcal K}}
\newcommand{\cB}{{\mathcal B}}
\newcommand{\cF}{{\mathcal F}}

\newcommand{\supp}{{{\rm Supp} \;}}
\newcommand{\tx}{{{\tilde x}}}
\newcommand{\tg}{{{\tilde \gamma}}}
\newcommand{\tm}{{{\tilde m}}}
\newcommand{\tH}{{{\tilde H}}}
\newcommand{\tHH}{{{\tilde {\mathcal H}}}}
\newcommand{\tI}{{{\tilde I}}}
\newcommand{\tII}{{{\tilde {\mathcal I}}}}
\newcommand{\tJ}{{{\tilde J}}}
\newcommand{\tJJ}{{{\tilde {\mathcal J}}}}

\newcommand{\tMk}{{{\tilde M}_k}}

\newcommand{\Psym}{{\bf P}_{\! sym}}
\newcommand{\PPP}{{\bf P}}
\newcommand{\MMM}{{\bf M}}

\begin{document}

\title{Propagation of chaos for the 2D viscous vortex model.}

\author{Nicolas Fournier, Maxime Hauray, St\'ephane Mischler}

\address{N. Fournier: LAMA UMR 8050, Universit\'e Paris Est,
Facult\'e de Sciences et Technologies,
61, avenue du G\'en\'eral de Gaulle, 94010 Cr\'eteil Cedex, France.}

\email{nicolas.fournier@univ-paris12.fr}

\address{M. Hauray: LATP UMR 7353, Universit\'e d'Aix-Marseille, Centre de
Math\'ematiques et Informatique, 39 rue F. Joliot Curie
13453 Marseille Cedex 13, France.}

\email{hauray@cmi.univ-mrs.fr}

\address{S. Mischler: CEREMADE UMR 7534, Universit\'e Paris-Dauphine,
Place du Mar\'echal de Lattre de Tassigny F-75775, Paris Cedex 16, France.}

\email{mischler@ceremade.dauphine.fr}



\thanks{The first and third authors of this work were supported  by a grant of the 
{\it Agence Nationale de la Recherche} numbered ANR-08-BLAN-0220-01.
}

\begin{abstract}
We consider a stochastic system of $N$ particles, usually called vortices in that setting,  approximating 
the 2D Navier-Stokes equation
written in vorticity. Assuming that the initial distribution of the position and
circulation of the vortices has 
finite (partial) entropy and a finite 
moment of positive order, we show that the empirical measure of the particle system
converges in law  to the unique (under suitable a priori estimates) solution of the 2D Navier-Stokes equation.
We actually prove a slightly stronger result : the propagation of chaos
of the stochastic paths towards the solution of the expected nonlinear 
stochastic differential equation.  Moreover, the convergence holds in a strong sense, usually called  
entropic (there is no loss of entropy in the limit). The result holds without
restriction (but positivity) on the viscosity parameter.

The main difficulty is the presence of the singular 
Biot-Savart kernel in the equation. 
To overcome this problem, we use the
dissipation of entropy which provides some (uniform in $N$)  bound on the Fisher
information of the particle system, and then use extensively that bound together
with classical and new properties of the Fisher information.
\end{abstract}

\maketitle


\section{Introduction}
\setcounter{equation}{0}
\label{sec:intro}

The subject of this paper is the convergence of a stochastic 
vortices system to the $2D$ Navier-Stokes equation
written in vorticity without restriction (but positivity) on the viscosity parameter. 

\vip

\paragraph{\bf The particle system.}
We consider a system of $N$ vortices labeled by an index $1 \le i \le N$,
each one being fully described
by its position $\XX^N_i \in \R^2$ and its circulation $ \frac1N \MM^N_i \in
\R$ which measures the ``strength'' of the vortices. We use what is called a 
{\it mean-field} 
scaling : the factor $\frac1N$ is there in order to keep the global circulation
(or also total vorticity) bounded. The case $\MM^N_i>0$ corresponds to a vortex
which turns round in the direct (trigonometric) sense while the 
case $\MM^N_i < 0$ corresponds  to a vortex which turns in the reverse sense. 
We assume that the system evolves stochastically according to the
following system of $\rr^2$-valued S.D.E.s 
on the vortices positions
\begin{align}\label{ps}
\forall i=1,\ldots,N,\quad \XX_i^N(t) = \XX^N_i(0)+  \frac1N \sum_{j \neq i} \MM^N_j\intot K(\XX_i^N(s) -
\XX_j^N(s))ds +  \sigma \BB_i(t),
\end{align}
where $((\BB_i(t))_{t\geq 0})_{i=1,\dots,N}$ stands for an independent family of $2D$ standard Brownian motions,
$\sigma > 0$ is a parameter linked to the viscosity and $K : \R^2 \to \R^2$ is
the Biot-Savart kernel defined by
$$
\forall \, x = (x_1,x_2) \in \R^2, \quad K(x)=\frac{x^\perp}{|x|^2} =  
\Bigl( -\frac{x_2}{|x|^2}, \frac{x_1}{|x|^2} \Bigr)
= \nabla^\perp \log |x|,
$$
It is worth emphasizing that we assume here that any vortex keeps its initial
circulation, so that 
the $\MM^N_i$  are time-independent and act like fixed parameters in \eqref{ps}.  
Each vortex's position
moves randomly according to a Brownian noise as well as deterministically according to a vector field
generated by all the other vortices
through the Biot-Savart kernel.
 The singularity of $K$ makes difficult the study of the particle system
(\ref{ps}). However, Osada \cite{o2} and others  have
shown that the particles a.s. never encounter,
so that the singularity of $K$ is a.s. never visited and (\ref{ps}) is
well-posed. See Theorem \ref{wpps} and  Section~\ref{sec:WP} below for more
details.

\vip

\paragraph{\bf The vorticity equation.}
As is now well-known, the dynamics of such models is linked to the
$2D$ Navier-Stokes equation written in 
vorticity formulation which will be called later the \emph{vorticity
equation}
with a viscosity $\nu = \sigma^2/2$
\begin{align} \label{NS2D}
\partial_t w_t(x) = (K \star w_t)(x)\cdot \nabla_x w_t(x) + \nu \Delta_x
w_t(x), 
\end{align}
where $w : \R_+ \times \R^2 \to \R $ is the vorticity function and the initial vorticity
$w_0 : \R^2 \to \R$ is given.
It is worth emphasizing again that we do not assume here that $w$ is non-negative and this is of 
course related to the fact that 
the circulations in the $N$-vortex system may be positive and negative.
There is a huge literature on that model. Our analysis will be based on the well-posedness 
of \eqref{NS2D} in a $L^1$-framework
as developed in Ben-Artzi in \cite{BenArtzi} and slightly improved in \cite{Brezis}. 
We also refer to Gallagher-Gallay \cite{gg} and the references 
therein for more recent well-posedness results.

\vip

\paragraph{ \bf Origin of the model.}
The deterministic $N$-particle system (with $\sigma =0$) was originally introduced by Helmholtz in 1858 
\cite{Helmholtz} 
and later studied by Kirchhoff \cite{Kirchhoff} and many others. It is sometimes quoted as the 
Helmholtz-Kirchhoff (HK in short) system.  In the non-viscous case $\sigma=0$,
the
empirical measure associated to the finite 
particle system (\ref{ps}) solves {\it exactly} the non viscous vorticity
equation (i.e. \eqref{NS2D} with 
$\sigma=0$) if the self interaction is neglected. This is not the case when
$\sigma>0$, where the Navier-Stokes equation is expected
to be solved only in the limit $N\to\infty$. Thus for a fixed $N$, addition of
noise on the position of the vortices as in \eqref{ps} is not the most relevant
idea. Physicists prefer to 
introduce damped 
vortices, known as Oseen vortices. 
The dynamics of Oseen vortices is still driven by the a variant of the deterministic HK system, where $K$
now varies with time, in order to take into account the dissipation. Both
systems (vortices driven by (\ref{ps}) or Oseen vortices) are 
interesting because they approximate the dynamics of real vortices, that appear,
for instance, in geostrophic or atmospheric flows, and are remarkably stable.
See the works of Marchioro  \cite{Marchioro} and Gallay \cite{Gallay} for a
justification of the approximation, and the one of Gallay-Wayne 
\cite{GallayWayne} for a precise mathematical result on the stability of Oseen vortices.

\vip

\paragraph{ \bf Interest of the limit in large number of vortices.}
Later, Onsager  \cite{Onsager} was the first to see the interest of the statistical properties 
of the $N$ vortices system to distinguish which among the numerous stationary solutions of the vorticity
equation are physically relevant.
His heuristic ideas where made more rigorous by Caglioti, Lions, 
Marchioro and Pulvirenti in \cite{CLMP1}. After that, the question of the convergence of
the HK model towards Euler and Navier-Stokes equations was studied by several authors. 
In the deterministic case Schochet  \cite{Schochet} proved the convergence towards solutions 
of the Euler equation. But since such solutions seem to be numerous under weak
a priori conditions, his results does not implies propagation of chaos.

As mentioned before, the stochastic model \eqref{ps} becomes much more relevant
when $N$ is large.
However, in that case the use of independent noise on each vortex is not 
motivated by the underlying physics. Since vortices are not real {\it particles} but rather small structures 
appearing in fluid models, a noise acting on them should depend on their relative positions: the noise
of two close vortices should be quite correlated. 
We refer to the work of Flandoli, Gubinelli and Priola 
\cite{FGP} for such more realistic models with a fixed number of vortices. But that kind of 
noise is much more difficult to handle in the limit of large number of vortices.
Thus in the sequel, we will use independent Brownian motion, despite this 
shortcoming.

A second interest of the stochastic vortices system is that it may be
seen as a companion model for stochastic particle systems
with positions and velocities, interacting
trough a two-body force and with velocities excited by independent brownian
motions. The case of a smooth interaction force has been extensively studied since the
work of McKean \cite{McK3}, but there is only few references in the  case of singular
interactions. Rather strangely, the deterministic case is better known since there
exist some convergence results for not too singular interaction without cut-off
\cite{HaurayJabin}. However, we do not know any result valid for singular
interaction
in the stochastic case. In fact such models seem tougher than the vortices
one, since the diffusion does not act on the full position-velocity variables.

\vip

\paragraph{\bf Main result.}
In order to connect the sequence of solutions   $\ZZ^N_t = (\MM_1, \XX^N_1(t), ... ,$ $ \MM_N,  \XX^N_N(t))$, 
to the random particle system \eqref{ps} with a solution to the vorticity
equation \eqref{NS2D}, we introduce the  vorticity  empirical measure 
$$
\WW^N_t (dx) := \frac 1N  \sum_{i=1}^N \MM_i^N \, \delta_{\XX^N_i(t)} (dx)
$$
which a.s. takes values in the space of bounded measures on $\R^2$, as well as the typical vorticity defined
from the law of $(\MM_1^N,\XX^N_1(t))$ in the following way 
$$
w^N_1(t,x) := \int_\R m \, \LLL (\MM_1^N,\XX^N_1(t)) (dm,x).
$$	
Then, under suitable (chaos)  hypothesis on the initial conditions $\ZZ^N_0$
we shall show that for any positive time 
\bear \label{eq:cvgce1}
&& \qquad \WW^N_t    \,\, \Rightarrow \,\, w_t \quad\hbox{in law as } \,\, N \to \infty  , 
\\ \label{eq:cvgce2}
&& \qquad w^N_1(t) \to w_t  \quad\hbox{strongly in} \,\, L^1(\R^2) \,\, \hbox{as} \,\, N \to \infty  ,
\eear
where $w$ is the unique solution to the vorticity
equation \eqref{NS2D} with appropriate initial datum $w_0$.

\vip 

In the other way round and in particular,   for any initial
(Lebesgue measurable) vorticity function $w_0 : \R^2 \to \R$ satisfying 
\beqn\label{condic}
\int_{\R^2}  |w_0 |\, ( 1 + | x |^k + |\log |w_0||) \, dx < \infty,
\quad
\text{for some} \quad k \in (0,2), 
\eeqn
we can build a sequence of initial conditions $\ZZ^N_0$ and then define the  family of solutions 
$\ZZ^N_t $ to the $N$-particle
vortex system \eqref{ps} so that the vorticity empirical measure $\WW^N_t$ and the typical vorticity 
$w^N_1(t)$  
converge to the solution $w_t$ as stated in  \eqref{eq:cvgce1} and  \eqref{eq:cvgce2}, where 
$w$ is the unique solution to the vorticity
equation \eqref{NS2D} with  initial datum $w_0$. 

The construction of the initial conditions is not very elaborated, but requires some notation that 
will be introduced in Section~\ref{sec:MainResult}.
Let us just mention that in the case of a non-negative initial vorticity, we can assume up to some 
scaling that $\omega_0$ is a probability, and then the choice 
$\MM_i^N =1$ for any $1 \le i \le N$ and $(\XX_i^N(0))_{1 \le i \le N}$ i.i.d. with law $\omega_0$ will do the job. 

\vip
\paragraph{\bf Chaos and limit trajectories.}

The solution $w$ of the vorticity equation is thus obtained as the limit in a kind of law of large numbers 
from the $N$-vortex system. 
However, the picture is not that simple. 

It is indeed rather reasonable to assume that the initial positions and circulations of
the vortices are (at least asymptotically) independent.
Then, as time passes, vortices interact and that create correlations so that vortices are never any longer 
independent. We may expect that these 
correlations vanish asymptotically because the interactions between pairs of vortices in \eqref{ps} tends 
to zero. 
For a smooth (say Lipschitz) interaction force field $K$, such a result is well-known since the pioneer 
work by McKean~\cite{McK3},
and it is related to the notions of {\it ``chaos"} and {\it ``propagation of chaos"} as  introduced by 
Kac in \cite{Kac1956}.  
Here, the Biot-Savard force field kernel is singular which leads to additional mathematical difficulties. 
Nevertheless, we are able to handle them and we  prove that still for the Biot-Savard kernel correlations 
vanish asymptotically. 

We establish this asymptotic independence and the convergence \eqref{eq:cvgce1}  
as a consequence of a stronger  {\it "trajectorial chaos"} that  we  briefly describe 
now. The method we follow 
 is closely related to the strategy introduced by Sznitman in \cite{S1} which  consists in showing that 
the sequence of empirical trajectories
$\mu^N_{\ZZ^N}$ converges to some stochastic process which is a solution to a nonlinear martingale problem.

\vip
Let us first notice that if we accept that correlations asymptotically disappear,  then the trajectories 
(and circulations)  $(\MM_i^N,(\XX_i^N(t))_{t\ge 0})_{i \le N}$ of the vortices must behave asymptotically 
like $N$ independent copies of the same process, 
$(\MM,(\XX(t))_{t\geq 0}$,  solution to the nonlinear stochastic differential equation
\beqn\label{NSDE}
\XX(t)=\XX(0)+  \intot \intrd K(\XX(s)-x)w_s(dx)ds + \sigma \BB_t,
\eeqn
where $w_t(dx)=  \int_{\rr\times B} m \, g_t(dm,dx)$ with $g_t = \LL(\MM,\XX(t))$.
It is important to stress that  $w_t$ solves necessarily the
vorticity equation (\ref{NS2D}) if $(\MM,(\XX(t))_{t\geq 0}$ is a solution to \eqref{NSDE}.

\vip
As a matter of fact, we will prove that under appropriate hypothesis on the initial law $\LL(\ZZ^N_0)$ 
(which includes chaos type assumption and bound on the entropy and on some
moment), the $N$-vorticity system enjoys a chaos property at the  level of the trajectories, namely,  
\beqn\label{eq:cvgce3}
\mu^N_{\ZZ^N} := \frac 1 N \sum_{i=1}^N \delta_{ (\MM_i^N,(\XX_i^N(t))_{t\ge 0}) } 
 \,\, \Rightarrow \,\,g  \quad\hbox{in law as } \,\, N \to \infty  , 
\eeqn
where $g$ is the law of the nonlinear process $(\MM,(\XX(t))_{t\geq 0}$ defined in \eqref{NSDE}. 
This convergence at the level of trajectories implies~\eqref{eq:cvgce1}.
Moreover, using a trick introduced in \cite{MMKacProg} which consists in carefully  
estimating  what happens for the dissipation
of the entropy, we deduce~\eqref{eq:cvgce2}.

\vip 
\paragraph{\bf An overview of the proof.}
Let us briefly describe our method, which relies on compactness/consistency/uniqueness
as in Sznitman in \cite{S1} who was studying the homogeneous Boltzmann equation. 
As already mentioned, the main difficulty comes from the fact 
that the kernel $K$ is singular so that the drift in 
\eqref{ps} may be very large when two particles are very close.
Using standard dissipation of entropy estimates, we obtain uniform bounds on the Fisher information of the 
time marginal of the law (of the positions) of the $N$ vortices. These uniform bounds provide enough 
regularity  to 

\begin{itemize}
\item[(1)] prove that {\it close encounters} of particles are rare, from which we deduce the 
tightness of the law of the trajectories of the $N$ vortices system (compactness),
\item[(2)] prove that the possible limit are made of solutions of the nonlinear
SDE (consistency), which satisfies some appropriate additional a priori bounds,
\item[(3)] prove the uniqueness of the above limit stochastic process.
\end{itemize} 

We may also remark that for the two first points (tightness and consistency), 
the singularity of the kernel in $1/|x|$ is not the critical one. Everything
would work for divergence free kernels with singularity behaving like $1/|x|^{\alpha}$,
with $\alpha \in (0,2)$. But, it is critical for the question of
uniqueness of the limit stochastic process and the Navier-Stokes equation.

\vip

We also emphasize the following important point. In order to get enough
{\it a priori} bounds on the possible limit, we use a new
result: the fact that the Fisher information (properly rescaled) can only
decrease when we perform the many-particle limit. This is
a consequence of the fact that the (so called) level $3$ Fisher information is linear
on mixed states. That property, known from the work of Robinson and Ruelle for
the entropy \cite{RR} was proved for the Fisher information by the two last
authors in \cite{hm}, with the help of the first author.
The precise result is properly stated and explained in Section~\ref{sec:prel2}.

\vip

We can also remark that our method is interesting only in low dimension
(of space). The extensive use of the Fisher information is very
interesting in this case, since it provides rather strong regularity. But,
if we increase the dimension, the regularity obtained from the Fisher
information gets weaker and weaker and we will not be able to treat
interesting singularities. As a matter of fact, it can be checked that our
method is valid for a divergence free kernel, with a singularity at most like
$1/|x|$ (included) near $0$ in any dimension $d$. Again, the limitation will
come from the uniqueness part, the tightness and  consistency will also
hold for singularity up to $1/|x|^2$, not included.

\vip

\paragraph{ \bf Already known results.}

If we replace the singular kernel $K$ by a regularized one $K_\e$,
the result of propagation of chaos is well-known. The more standard strategy is
due to McKean
\cite{McK3} and applies when the interaction is Lipschitz. It relies on a coupling
argument between the solution of the $N$-vortices stochastic system and $N$
independent copies of the solution of the nonlinear SDE. But since the result
gives a quantitative estimate of convergence, an optimization may lead to a
similar result valid for a regularization parameter $\e$ going to $0$ with $N$.
That approach or some variants was performed by Marchioro and Pulvirenti in
\cite{MarPul2}, for bounded initial vorticity. For that regularity on the
initial condition, there is a good well-posedness theory even in the
non-viscous case, so that their method also applies if $\nu=0$. The drawback is
that the speed of convergence of the
regularization parameter is very slow : $\e(N) \sim \log(N)^{-1}$.
See also M\'el\'eard \cite{m2} for a similar result for more general initial
data.

It is worth emphasizing that the convergences \eqref{eq:cvgce1},  \eqref{eq:cvgce2},  \eqref{eq:cvgce3}
are proved without any rate. This is a consequence of the compactness method we use. 
In particular, we were not able to implement the coupling method popularized by Sznitman~\cite{SSF} 
and revisited by Malrieu \cite{Malrieu}, see also \cite{BGG} and the references therein for recent 
developments, 
nor the quantitative Grunbaum's duality method elaborated in \cite{MMW}.

\vip

In a series of papers, Osada proves the convergence of the particle system (\ref{ps}) to the vorticity
equation (\ref{NS2D}): the case of a large viscosity is studied in \cite{o0}
and the case of any positive viscosity is discussed in \cite{Osada}. In this last paper,
the pathwise convergence is not obtained (while it is checked in \cite{o0} when $\sigma$ is large enough).
His strategy relies strongly on a deep result obtained by himself in \cite{oDiff}: 
estimates {\it \`a la Nash} for convection-diffusion equation, with 
divergence free and very singular drift. This last result is also a
key argument in most works about existence and uniqueness for the $2D$ Navier-Stokes equation, 
with the exception of the work of Ben-Artzi \cite{BenArtzi} that we use here.

\vip

Let us finally mention the result of Cepa and Lepingle
\cite{CepaLepingle} about
Coulomb gas models in dimension one. Their models are very similar to ours,
but their singularity is repulsive and strong since it behaves like
$1/|x|$ (far above the singularity of the Coulomb law since we are in
dimension one). However, their technics are limited to dimension one.

\vip

The present paper improves on preceding results in several directions. 
It does not require that the viscosity coefficient is large as in Osada
\cite{o1}, nor 
to cutoff the interaction kernel in the particle system as in \cite{MarPul2} or
\cite{m2}. Moreover, 
in the two above mentioned previous works of Osada, the convergence \eqref{eq:cvgce2} was 
only established in the weak sense (of measures) and only for non-negative
vorticity. Moreover, the results of Osada basically apply when $w_0 \in
L^\infty$, 
while we allow any $w_0 \in L^1(\rr^2)$ with a finite entropy and moment of
positive order.
Last but not least,  our proof seems simpler than  the one  of Osada in \cite{o1}, which uses very technical estimates.

\vip

\paragraph{ \bf The case of bounded domains.}
In the case of general bounded domains $\Omega$ with boundaries, the problem
is more delicate. The first difficulty is that the vorticity formulation 
of the Navier-Stokes equation does not behave well with the boundaries
conditions. In fact, vorticity is created at the boundary. However, 
it is still possible to imagine branching processes of interacting particles
that will take the possible creation and annihilation of vortices at the
boundary, as is done by Benachour, Roynette and Vallois in \cite{BeRoVa}, but
the analysis of such systems seems much more difficult. 

\vip

However, if we move to some periodic and bounded setting, $\Omega = \T^2$, then
our results will apply with small modifications. All we have to do is to
replace the Biot-Savard $K$ by its periodization
$$
K_{per}(x)  := \frac{x^\perp}{|x|^2} + g_\infty(x)
$$
where $g_\infty$ is some $C^\infty$ function. The singularity is exactly the
same, and the addition of a smooth function $g_\infty$ in the kernel does not
  raise any difficulty. As a consequence,  our result will apply to that case with the
appropriate modifications.

\section{Statement of the main results}
\setcounter{equation}{0}
\label{sec:MainResult}

\subsection{Notation}

For any Polish space $E$, we denote by $\PPP(E)$ the set of probability measures on $E$
and by $\MMM(E)$ the set of finite signed measures on $E$.
Both are endowed  with the topology of weak convergence defined by duality against functions of
$C_b(E)$.
For $N\geq 2$, we denote by $\Psym(E^N)$ the set of symmetric probability measures $F$ on $E^N$ 
(i.e. such that $F$ is the law of an exchangeable $E^N$-valued random variable $(\YY_1,\dots,\YY_N)$).

\vip

In the whole paper, when $f\in\MMM(\rr^d)$ has a density, we also denote by $f\in L^1(\rr^d)$
its density.

\vip

For $x\in \rr^2$, we introduce $\langle x \rangle:=(1+|x|^2)^{1/2}$.
For $k \in (0,1]$ and $N \ge 1$ we set  
$$
\forall \, X = (x_1, ..., x_N) \in (\R^2)^{N}, \quad
\langle X \rangle^k := \frac 1 N \sum_{i=1}^N \langle x_i \rangle^k.
$$
For $F \in \PPP ((\R^2)^{N})$, we define 
$$
M_k (F) := \int_{(\R^2)^N} \langle X \rangle^k \, F(dX).
$$
We also introduce 
$$
\PPP_k ((\R^2)^N) := \{F \in \PPP((\R^2)^N) \; : \;  M_k(F) < \infty \}.
$$
For $F \in \PPP((\R^2)^N)$ with a density (and a finite moment of positive
order for the entropy), we introduce the Boltzmann entropy and the Fisher
information of $F$ defined as
$$
H(F):= \frac1N\int_{(\rr^2)^N} F(X)\log(F(X)) \, dX \quad \hbox{and}\quad
I(F):= \frac1N\int_{(\rr^2)^N} \frac{|\nabla F(X)|^2}{F(X)} \, dX. 
$$
If $F \in \PPP((\R^2)^N)$ has no density, we simply put $H(F) = +\infty$ and $I(F) = +\infty$. 
The somewhat unusual normalization by $1/N$ is made in order that for any $f \in \PPP(\rr^2)$,
$$
H(f^{\otimes N}) = H(f)  \quad \hbox{and}\quad
I(f^{\otimes N}) = I(f).
$$

We will often deal here with probability measures on $(\R \times \R^2)^N$, 
representing the circulations and positions of $N$ vortices. But the
circulations 
only act like parameters. We thus adapt all the previous notation by a simple integration.
For $G \in \PPP ((\R \times \R^2)^N)$, write the disintegration 
$G(dM,dX)=R(dM)F^M(dX)$, where $R \in\PPP(\rr^N)$ and for each $M\in \rr^N$,
$F^M \in \PPP((\rr^2)^N)$ and define partial moment, entropy and Fisher
information by
\begin{eqnarray}
\tMk(G) &:= &\int_{\rr^N} M_k(F^M) \, R(dM) = \int_{(\R \times \R^2)^N }
\langle X \rangle^k \, G(dM,dX) \label{dftmk}\\
 \tH(G) &:= & \int_{\rr^N} H(F^M)
\, R(dM)  \label{dfth}\\
\tI(F) &:= & \int_{\rr^N} I(F^M) \, R(dM). \label{dfti}
\end{eqnarray}
To understand these objects, let us make a few observations. 
When $G$ has a density on $(\rr\times\rr^2)^N$,
$$\tI(F)=  N^{-1} \int_{(\R \times \R^2)^N } \frac{|\nabla_{_X}
G(M,X)|^2}{G(M,X)} dMdX.$$
When $G$ has a finite (classical) entropy, we can write
$$\tH(G)= \int_{(\R \times \R^2)^N} G(M,X)\log G(M,X) dMdX  - \int_{\R^N}
R(M)\log R(M) dM
= H(G) - H(R).$$

\vip

We finally introduce 
$$
\PPP_k ((\R \times \R^2)^N) := \{G \in \PPP((\R \times \R^2)^N); \,\, \tilde M_k(G) < \infty \}  .
$$

\subsection{Notions of chaos}

In this subsection, $E$ will stand for an abstract polish space. 

\begin{defi}[Chaos for probability measures]
A sequence $(F^N)$ of symmetric probability measures on $E^N$ is said to be  
$f$-chaotic, for a probability measure $f$ on $E$ if one of three
following equivalent conditions is satisfied:

(i) the sequence of second marginals $F^N_2 \wto f \otimes f$ as $N \to + \infty$;

(ii) for all $j \ge 1$, the sequence of $j$-th marginals $F^N_j \wto f^{\otimes j}$ as $N \to + \infty$;

(iii) the law $\hat F^N$ of the empirical measure (under $F^N$) converges towards $\delta_f$ in 
$\PPP(\PPP(E))$ as $N\to\infty$. 
\end{defi}

This definition translates into an equivalent definition in terms of random variables.

\begin{defi}[Chaos for random variables]
A sequence $(\YY^N_1,\ldots,\YY^N_N)$ of exchangeable $E$-valued random variables
is said to be $\YY$-chaotic for some $E$-valued random variable $\YY$ if the
sequence of laws $\LL(\YY^N_1,\ldots,\YY^N_N)$ is $\LL(\YY)$-chaotic, in other words,
if
one of three following equivalent condition is satisfied:

(i) $(\YY^N_1,\YY^N_2)$ goes in law to $2$ independent copies of $\YY$
as $N\to\infty$;

(ii) for all $j \ge 1$, $(\YY^N_1, \dots, \YY^N_j)$ goes in law to $j$ independent copies of $\YY$
as $N\to\infty$;

(iii) the empirical measures $\mu^N_{\YY^N} = \frac1N \sum_1^N \delta_{\YY_i^N}
\in \PPP(E)$ go in law to the constant
$\LL(\YY)$ as $N\to\infty$.
\end{defi}

We refer for instance to the lecture of Sznitman \cite{SSF} for the equivalence of the three conditions,
as well as \cite[Theorem 1.2]{hm}Ê
where that equivalence is established in a {\it quantitative way}. 
Let us only mention that exchangeability is very important in order to understand point (i).

\vip

Propagation of chaos in the sense of Sznitman holds for a system $N$ exchangeable particles 
evolving in time (for instance the system \eqref{ps}) if when the initial conditions 
$(\YY^{N}_1(0),\ldots,\YY^{N}_N(0))$ are $\YY(0)$-chaotic, the trajectories 
$((\YY^N_1(t))_{t\geq 0},\ldots,(\YY^N_N(t))_{t\geq 0})$ are $(\YY(t))_{t\ge 0}$-chaotic, where 
$(\YY(t))_{t\ge 0}$ is the 
(unique) solution of the expected (one-particle) limit model (here the nonlinear SDE (\ref{NSDE})). 

\vip 

Another (stronger) sense of chaos has been developped: the entropic 
chaos. 
It goes back to a celebrated work of Kac \cite{Kac1956} and was formalized recently in \cite{CCLLV,hm}
(see also \cite{hm} for a notion of Fisher information chaos). 

\begin{defi}[Entropic chaos] \label{def:EntChaos}
A sequence $(F^N)$ of symmetric probability measures on $E^N$ is said to be  
entropically $f$-chaotic, for a probability measure $f$ on $E$, if
$$
F^N_1 \to f \quad \hbox{weakly in } \PPP(E) \quad \hbox{and} \quad H(F^N) \to H(f)
$$
as $N\to \infty$, where $F^N_1$ stands for the first marginal of $F^N$. 
\end{defi}
It is shown in \cite{hm} that this is in fact a stronger notion than propagation of chaos. 
Actually, it is known that the entropy can only decrease if a sequence $F^N$ is $f$ chaotic: 
we  say that the entropy is $\Gamma$-lower semi continuous.
With our normalization, it writes
$$
H(f)  \le \liminf_{N \to \infty} H(F^N).
$$
Since the entropy is convex, $\lim H(F^N) =H(f)$ is a stronger notion of convergence, 
which implies that 
for all $j\geq 1$, the density of the law of $(\YY^N_1, \dots, \YY^N_j)$ goes to $f^{\otimes j}$
strongly in $L^1$. 

Here, we will have to modify slightly this notion, replacing the use of $H$ by that of $\tH$,
since the circulations
of the vortices only act like parameters.

\subsection{The Navier-Stokes equation}

\begin{defi}\label{def:weakNS}
We say that $w = (w_t)_{t\geq 0} \in  C([0,\infty),\MMM(\rr^2))$ 
is a weak solution to (\ref{NS2D}) if 
\begin{equation}\label{cfas}
\forall \; T>0,\quad \int_0^T \intrd\intrd |K(x-y)| \,  |w_s|(dx) \,  |w_s|(dx) \, ds<\infty
\end{equation}
and if for all $\varphi \in C^2_b(\rr^2)$, all $t \geq 0$,
\begin{align}\label{ws}
\intrd \varphi(x)w_t(dx)=& \intrd \varphi(x)w_0(dx)  
+ \intot \intrd \intrd K(x-y) \cdot \nabla \varphi(x) \, w_s(dy) \,  w_s(dx) \, ds\\
&+  \nu
\intot\intrd \Delta\varphi(x)w_s(dx) \, ds. \nonumber
\end{align}
\end{defi}

We will establish the following extension of \cite{BenArtzi,Brezis} 
which is well adapted to our purpose. 

\begin{thm}\label{th:wp}
Assume that $w_0 \in L^1(\R^2)$ satisfies (\ref{condic}). There exists a unique weak solution $w$
to (\ref{NS2D}) such that   
\begin{align}
\label{cfish2}
&\nabla_x w \in L^{2q/(3q-2)}(0,T, L^q(\rr^2))\quad  \forall \; q\in [1,2),\quad
\forall \; T>0.
\end{align}
This solution furthermore satisfies
\beqn\label{bdd:dbK3}
w  \in C([0,\infty) ; L^1(\rr^2)) \cap C((0,\infty) ; L^\infty(\rr^2))
\eeqn
and
\beqn\label{eq:betaw}
\partial_t \beta(w) =  (K* w) \cdot \nabla_x \beta(w) +  \nu \Delta
\beta(w) -  \nu \beta''(w) \, |\nabla w|^2 \quad\hbox{on}\quad [0,\infty)
\times \R^2 
\eeqn
in the distributional sense, 
for any $\beta \in C^1(\rr) \cap W^{2,\infty}_{loc}(\R)$ such that $\beta''$ is piecewise continuous and  
vanishes outside of a compact set.
\end{thm}

As we will see in the proof of the above result, thanks to the Sobolev embedding and  
the Hardy-Littlewood-Sobolev inequality
one can show that for  $w\in  C([0,\infty),\MMM(\rr^2))$,
\eqref{cfish2} implies \eqref{cfas}.
The proof of \eqref{eq:betaw} is classical. When $\nu=0$ such a  result has been
proved in 
\cite[Theorem II.2]{dPL}
while the case $\nu > 0$ can be obtained by adapting a result from
\cite[Section III]{dPL-BFP}. 
For the sake of completeness, we will however sketch 
the proof of \eqref{eq:betaw} in Section~\ref{sec:blm}.

\subsection{Stochastic paths associated to the Navier-Stokes equation}

Since a solution $(w_t)_{t\geq 0}$ to the vorticity equation 
(\ref{NS2D}) does not take values in probability measures on $\rr^2$,
a few work is needed to find some related stochastic paths:
roughly, we write the initial vorticity $w_0$
as some partial information of the law $g_0$ of circulations and positions of
the vortices.

\vip

We consider $g_0 \in \PPP(\rr\times\rr^2)$ satisfying, for some
$A\in (0,\infty)$ and some $k\in (0,1]$,
\begin{align}\label{condic2}
\supp g_0 \subset \AA \times \rr^2, \;\;  \AA=[-A,A], \;\; \tMk(g_0)<\infty  
\;\; \hbox{and}\;\; \tH(g_0)<\infty.
\end{align}

\begin{rk}\label{remtc}
(i) Let $g_0 \in \PPP(\rr\times\rr^2)$ satisfying (\ref{condic2}) and define $w_0 \in \MMM(\rr^2)$ 
by
\beqn\label{lienwg}
\forall\; B \in \cB(\rr^2), \quad w_0(B)=\int_{\rr\times B} m g_0(dm,dx).
\eeqn
Then $w_0 \in L^1(\rr^2)$ and satisfies (\ref{condic}).

(ii) For $w_0 \in L^1(\rr^2)$ satisfy (\ref{condic}), it is possible to find a probability measure $g_0$ on 
$\rr\times\rr^2$ satisfying (\ref{condic2}) and such that (\ref{lienwg}) holds true.
\end{rk}

\begin{proof}
We first check (i). First, $|w_0|(dx) \leq A \int_{m\in \rr} g_0(dm,dx)$, so that
$|w_0|(\rr)\leq A$ and $w_0$ is a finite measure. Next, there holds $\int_{\rr^2} \langle x \rangle^k |w_0|(dx)
\leq A  \int_{\rr \times \rr^2} \langle x \rangle^k  g_0(dm,dx)<\infty$. Finally, to prove that $|w_0|$
has a density satisfying $\int_{\rr^2} |w_0(x)|\log (|w_0(x)|)dx <\infty$, it obviously suffices
to check that $\kappa(dx):=\int_{m\in \rr} g_0(dm,dx)$ has a finite entropy, since $|w_0| \leq A \kappa$.
We thus disintegrate $g_0(dm)=r_0(dm)f^m_0(dx)$ and use the convexity of the entropy functional
to get 
$$H(\kappa)= H\left(\int_{m\in \rr} r_0(dm)f^m_0(dx)\right) \leq \int_\rr r_0(dm) H(f^m_0)=\tH(g_0)<\infty
\quad \text{by assumption.}
$$ 

To verify (ii), 
write $w_0=w^+_0 - w^-_0$, for two non-negative functions with disjoint 
supports $w_0^+$ and $w_0^-$, put $a:=\intrd |w_0(x)|dx$ and set (for example)
$$
g_0(dm,dx)= \frac1{a} \delta_{a}(dm)  w_0^+(x)dx+ \frac1{a} \delta_{-a}(dm) w_0^-(x)dx.
$$
Then (\ref{lienwg}) holds true and (\ref{condic2}) is easily deduce from (\ref{condic}).
This is the most simple possibility, but many other exist. In general, $g$ may
be seen as a Young measure associated to $w$, and it may be of physical
interest to introduce Young measures in the context of the Euler equation, see for
instance \cite{Bouchet}.
\end{proof}

We can now introduce some (stochastic) paths associated to the
vorticity equation.

\begin{defi}\label{edsnl}
Let $g_0$ be a probability measure on $\rr\times \rr^2$ and consider a $g_0$-distributed random
variable $(\MM,\XX(0))$ independent of a $2D$-Brownian motion $(\BB_t)_{t\geq 0}$.
We say that a  $\rr^2$-valued process $(\XX(t))_{t\geq 0}$ solves the nonlinear SDE \eqref{NSDE} if
for all $t\geq 0$,
\begin{align} \label{nlinpro}
\XX(t)=\XX(0)+  \intot \intrd K(\XX(s)-x) \, w_s(dx)ds + \sigma \BB_t,
\end{align}
where $w_t$ is the measure on $\rr^2$ defined by 
\beqn\label{def:vorticityOFg}
\forall\; B \in \cB(\rr^2), \quad w_t(B)= \E[\MM \indiq_{\{\XX(t)\in B\}}] =  \int_{\rr\times B} m \, g_t(dm,dx)
\eeqn
where $g_t = \LL (\MM,\XX(t))$,  and if $(w_t)_{t\geq 0}$ satisfies (\ref{cfas}).
\end{defi}

Roughly, $\MM$ represents the circulation of a typical vortex and
$(\XX(t))_{t\geq 0}$ its path,
in an infinite vortices system subjected to the vorticity equation. The rigorous
link is the following.

\begin{rk}\label{nsdeins2d}
For $(\XX(t))_{t\geq 0}$ a solution to \eqref{NSDE}, $(w_t)_{t\geq 0}$ is a weak solution to (\ref{NS2D}).
\end{rk}

\begin{proof}
This can be checked by an application of the It\^o formula: for $\varphi\in C^2_b(\rr^2)$, we have
\bean
\MM\varphi(\XX(t))&=&\MM\varphi(\XX(0))+ \intot \intrd \MM \nabla\varphi(\XX(s)) \cdot K(\XX(s)-x)w_s(dx)ds \\
&&+ \,\nu
\intot \MM \Delta\varphi(\XX(s))ds + \sigma \intot \MM \nabla\varphi(\XX(s))
d\BB_s,
\eean
where we recall that $\nu := \sigma^2/2$. Taking expectations
and using that the last term is a martingale with mean $0$, we find (\ref{ws}).
\end{proof}

We will check the following consequence of Theorem~\ref{th:wp}.

\begin{thm}\label{wpnsde}
Let $g_0$ be a probability measure on $\rr\times \rr^2$ satisfying (\ref{condic2}). There exists a unique 
strong solution $(\XX(t))_{t \ge 0}$ to the nonlinear SDE \eqref{NSDE} 
such that
\begin{align}\label{tif}
\forall \; T>0,\quad \int_0^T \tI(g_s)ds <\infty,
\end{align}
$g_t\in \PPP(\rr\times \rr^2)$ being the law of $(\MM,\XX(t))$.
Furthermore, its associated vorticity function $(w_t)_{t\geq 0}$ satisfies \eqref{cfish2} and $(g_t)_{t\geq 0}$
satisfies the entropy equation 
\beqn\label{eq:EntropyEq}
\tilde H(g_t) + \nu \int_0^t \tilde I(g_s) \, ds = \tilde H(g_0) 
\qquad \forall \, t > 0.
\eeqn
\end{thm}

\subsection{The stochastic particle system} \label{sec:WP}

As shown by Osada \cite{o2} and others, the system  \eqref{ps} is well-posed.

\begin{thm}\label{wpps}
Consider any family $(\MM_i^N,\XX_i^N(0))_{i=1,\dots,N}$ of
$\rr\times\rr^2$-valued random variables, independent of a family $(\BB_i(t))_{i=1,\dots,N,t\geq 0}$ 
of i.i.d. $2D$-Brownian motions and such that a.s., $\XX_i^N(0)\ne \XX_j^N(0)$ for all $i\ne j$. 
There exists a unique strong solution to (\ref{ps}).
\end{thm}

Actually, Osada \cite{o2} shows that a.s., for all $t\geq 0$, all $i\ne j$, $\XX^N_i(t)\ne \XX^N_j(t)$.
This implies the well-posedness of (\ref{ps}), since the singularity of $K$ is thus a.s.
never visited by the system.

\vip

Let us give a few more references. When the circulations 
$\MM_i^N$ are positive, Takanobu proved  the well-posedness of the system using a martingale argument 
\cite{Takanobu}.  Osada extended in \cite{o2} his results to arbitrary
vorticities  using estimates 
{\it \`a la Nash}
for fundamental solutions to parabolic equations with divergence free drift \cite{oDiff}. More recently, 
Fontbona and Martinez adapted in \cite{FontMart} the technique used by Marchioro
and Pulvirenti for the 
deterministic $N$ vortex models \cite[Chapter 4.2]{MarPul} to the stochastic case \eqref{ps}.

\subsection{The result of propagation of chaos}
To study the many-particle limit of the vortex system (\ref{ps}), we  have to impose some compactness and
consistency properties on the initial system.

\vip

Denote by $G^N_0 \in \PPP((\R\times\R^2)^N)$ the law of $(\MM_i^N,\XX_i^N(0))_{i=1,\dots N}$. We will assume that 
there are $k\in (0,1]$, $A\in (0,\infty)$ and  $g_0\in \PPP(\R\times\R^2)$ supported in $\AA\times\R^2$, where
$\AA=[-A,A]$, such that, setting $r_0(dm):=\int_{x\in\rr^2}g_0(dm,dx) \in \PPP(\AA)$,

\beqn\label{chaosini}
\left\{\begin{array}{l}
G^N_0 \in  \Psym((\rr\times\rr^2)^N) \hbox{ is } g_0\hbox{-chaotic;} \\ \\
\sup_{N\geq 2} \tMk(G^N_0)<\infty, \quad \sup_{N\geq 2}\tH(G^N_0) <\infty;  \\ \\
R^N_0(dm_1,\dots,dm_N) := \int_{(\R^2)^N} G^N_0(dm_1,dx_1,\dots,dm_N, dx_N) = r_0^{\otimes N}(dm_1,\dots,dm_N).
\end{array}\right.
\eeqn

This last condition asserts that $\MM_1^N,\dots,\MM_N^N$ are i.i.d. and $r_0$-distributed.

\begin{rk}\label{remc}
(i) The typical situation is the following: let $g_0 \in \PPP(\rr\times\rr^2)$ satisfy 
\eqref{condic2} and consider, for $N\geq 2$, an i.i.d. family
$(\MM_i^N,\XX_i^N(0))_{i=1,\dots,N}$ of $g_0$-distributed random variables. 
Then $G_0^N=g_0^{\otimes N}$ and \eqref{chaosini} is met.

(ii) Consider a family $(\MM_i^N,\XX_i^N(0))_{i=1,\dots N}$ satisfying (\ref{chaosini})
with some $g_0 \in \PPP(\R\times\R^2)$. Then $g_0$ automatically satisfies (\ref{condic2}),
so that the nonlinear SDE \eqref{NSDE} has a unique solution associated to $g_0$ by Theorem \ref{wpnsde}. 
Also, $w_0$ defined from $g_0$ as in (\ref{lienwg}) satisfies (\ref{condic}) by 
Remark \ref{remtc}-(i), so that
Theorem \ref{th:wp} implies that (\ref{NS2D}) with $w_0$ as initial condition is well-posed.

(iii) Under \eqref{chaosini}, we have $\tH(G^N_0)<\infty$ for each $N\geq 2$, whence the law
of $(\XX_1^N(0),\dots,\XX_N^N(0))$ has a density on $(\rr^2)^N$. In particular, $\XX_i^N(0)\ne \XX_j^N(0)$
a.s. for all $N\geq 2$, all $i\ne j$, so that for each $N\geq 2$, the particle system \eqref{ps} is 
well-posed by Theorem 
\ref{wpps}.
\end{rk}

\begin{proof}
Point (i) is easily checked, using in particular that $\tMk(G^N_0)=\tMk(g_0)$
and $\tH(G^N_0)=\tH(g_0)$
for all $N\geq 2$. 
For (ii), we just have to check that $g_0$ satisfies (\ref{condic2}). 
But by exchangeability we have $M_k(G^N_0)= M_k(G^N_{0,1})$, where
$G^N_{0,1}$ denotes the first marginal of $G^N_0$. Since $ G^N_{0,1} \wto g_0$
weakly in the sense of measures, we get
$\tMk(g_0) \leq \liminf_N \tMk(G^N_{0,1}) <  \infty$.
Finally, the $\Gamma$-sci property $\tH(g_0) \leq \liminf_N \tH(G^N_0)$ is more 
difficult to prove but follows from Theorem \ref{th:levl3tH&tI} below.
Point (iii) is obvious.
\end{proof}

Let us now write down the first part of our main result, concerning the paths of
the particles.
Here $C([0,\infty),\rr^2)$ is endowed with the topology of 
uniform convergence on compacts.

\begin{thm}\label{th:mr}
Consider, for each $N\geq 2$, a family $(\MM_i^N,\XX_i^N(0))_{i=1,\dots,N}$
of $\rr\times\rr^2$-valued random variables. Assume that the initial chaos assumptions \eqref{chaosini} holds 
true for some $g_0$.
For each $N\geq 2$, consider the unique solution (see Remark \ref{remc}-(iii))
$(\XX^N_i(t))_{i=1,\dots,N,t\geq 0}$ to (\ref{ps}), and the unique solution $(\XX(t))_{t\ge 0}$ to 
the nonlinear SDE \eqref{NSDE} given by Theorem~\ref{wpnsde} associated to
$g_0$ (see
Remark \ref{remc}-(ii)). Then, the sequence $(\MM^N_i,(\XX^N_i(t))_{t\ge 0})_{i = 1 \ldots,N}$ is 
$(\MM,(\XX(t))_{t\ge 0})$-chaotic.

In particular, it implies that if we set
\begin{equation}\label{def:WN}
 \WW^N_t:= \frac1N \sum_{i=1}^N \MM_i^N \delta_{\XX_i^N(t)},
\end{equation}
then $(\WW^N_t)_{t\geq 0}$  goes in probability in $C([0,\infty),\MMM(\rr^2))$, as $N\to \infty$, to
the unique weak solution
$(w_t)_{t\geq 0}$ given by Theorem~\ref{th:wp} to the vorticity equation
(\ref{NS2D}) starting from $w_0$  (see Remark \ref{remc}-(ii)).
\end{thm}

Our last result deals with entropic chaos.

\begin{thm}\label{th:mr2}
Adopt the same notation and assumptions as in Theorem~\ref{th:mr} and assume furthermore that
$\lim_n \tH(G^N_0)=\tH(g_0)$ (which is the case if $G^N_0 = g^{\otimes N}$). For $t\geq 0$, denote by
$g_t \in \PPP(\rr\times\rr^2)$ the law of  $(\MM,\XX(t))$.

(i) For all $t\geq 0$, $((\MM_i^N,\XX_i^N(t))_{i=1,\dots,N}$ is $g_t$-entropically chaotic in the sense that,
denoting by $G^N_t\in\PPP((\R\times \R^2)^N)$ its law,
$$
(\MM_1^N,\XX_1^N(t)) \to  g_t \quad \hbox{in law} \quad \hbox{and} \quad \tH(G^N_t) \to \tH(g_t)
\quad \hbox{as}\quad N\to\infty.
$$

(ii) For $j=1,\dots,N$, define the $j$-particle vorticity $w^N_{jt}$ as the measure
on $(\rr^2)^j$:
\beqn\label{eq:TypicalVorticity}
w^N_{jt} (dx_1, \dots , dx_j) := 
\int_{m_1,\dots,m_N \in \rr, \; x_{j+1},\dots,x_N \in \rr^2} m_1 \dots m_j  \, G^N_t(dm_1,dx_1, \dots, dm_N,dx_N).
\eeqn
This measure has a density and for all fixed $t\geq 0$, all fixed $j\geq 1$, 
\beqn\label{eq:cvgewN1fort}
w^N_{jt} \to w_t^{\otimes j}  \quad\hbox{strongly in} \,\, L^1((\R^2)^j) \,\, \hbox{as} \,\, N \to \infty.
\eeqn
\end{thm}

\subsection{Plan of the proof}

In Section \ref{sec:prel}, we prove various functional inequalities, showing in particular that 
a Fisher information estimate for the $N$-particles distribution allow us 
to control the {\it close encounters}  between particles.
Section \ref{sec:prel2} is dedicated to a result in the spirit of
Robinson-Ruelle \cite{RR}: the partial entropy $\tH$ and partial 
Fisher information $\tilde I$ are affine on mixed states, which implies the
$\Gamma$-lower semi continuity 
of both functionals. Precisely, that was proved in \cite{hm} for the full
entropy and Fisher information and here we only present the adaptation necessary
in the partial case. In Section \ref{sec:smainest}, we prove our main estimate: 
denoting by $G^N_t=\LLL((\MM^N_1,\XX^N_1(t)),\dots,(\MM^N_N,\XX^N_N(t)))$, 
$$
\forall \; T>0, \quad
\sup_{N\geq 2} \left\{ \sup_{[0,T]}[\tilde H(G^N_t) + \tilde M_k(G^N_t) ]+ \int_0^T \tilde I(G^N_t) \, dt  
\right\} <\infty
$$
and deduce the tightness of our system. We then show that any limit point solves the nonlinear S.D.E.
in Section \ref{sec:pr}, and satisfies the a priori condition of
Theorem~\ref{wpnsde}. We prove our uniqueness results (Theorems \ref{th:wp} and
\ref{wpnsde}) 
in Section \ref{sec:blm}
and conclude the proofs of Theorems 
\ref{th:mr} and \ref{th:mr2} in
Section \ref{sec:concl}.

\vip
We close that section with a convention that we shall use in all the sequel. 
We write $C$ for a (large) finite constant and $c$ for a positive constant depending
only on $\sigma$ and on all the bounds assumed in (\ref{condic}), (\ref{condic2}) and (\ref{chaosini}). 
Their values 
may change from line to line. All other dependence will be indicated in subscript.

\section{Entropy and Fisher information} \label{sec:prel}
\setcounter{equation}{0}

In this section, we present a series of results involving the Boltzmann
entropy $H$, the Fisher information $I$ and their modified versions $\tH,\tI$.
In the sequel of the article, they will provide key estimates in order to
exploit the regularity of the objects we will deal with. 

\vip

The following very classical estimate will be useful in order to get bounds on the system of particles
in the next section. It also explains why the entropy is well-defined from
$\PPP_k((\R^2)^N)$ into $\R \cup \{+\infty \}$. 
See the comments before \cite[Lemma 3.1]{hm} for the proof.

\begin{lem}\label{ieth}
For any $k,\lambda \in (0,\infty)$, there is a constant $C_{k,\lambda} \in \R$ such that 
for any $N\geq 1$, any $F \in \PPP_k((\rr^{2})^N)$
$$
H (F) \geq  - C_{k,\lambda} - \lambda \, M_k(F).
$$
\end{lem}

We next establish some kind of Gagliardo-Nirenberg-Sobolev inequality involving the Fisher information.

\begin{lem} \label{lem:FishInteg1}
For any $f \in \PPP(\R^2)$ with finite Fisher information, there holds 
\begin{align} 
\label{eq:LpbdFisher} 
&\forall \; p\in [1,\infty), \quad \| f \|_{L^p(\rr^2)} \leq C_p \,   I(f)^{1 - 1/p},\\
\label{eq:nablaLqbdFisher}
&\forall \; q\in [1,2), \quad
\| \nabla f \|_{L^q(\rr^2)} \leq C_q  \,  I(f)^{ {3/2}-{1/q}}.
\end{align}
\end{lem}

\begin{proof} We start with (\ref{eq:nablaLqbdFisher}).
Let $q \in [1,2)$ and use the H\"older inequality:
$$
\|\nabla f\|_{L^q}^q = \int \left|\frac{\nabla f}{\sqrt f}\right|^q f^{q/2}
\leq \left( \int \frac{|\nabla f|^2}{ f} \right)^{q/2} 
\left( \int  f^{q/(2-q)} \right)^{(2-q)/2} =  I(f)^{q/2} \, \|f \|_{L^{q/(2-q)}}^{q/2}. 
$$
Denoting by $q^\ast  =  2q/(2-q) \in [2,\infty)$ the Sobolev exponent associated to $q$, we have,
thanks to a standard interpolation inequality and to the Sobolev inequality,
\begin{equation} \label{eq:Lq*bdFisher}
 \|f \|_{L^{q/(2-q)}} =  \|f \|_{L^{q^\ast / 2}} \le \| f \|_{L^1}^{1/ (q^\ast - 1)} \, 
\| f \|_{L^{q^\ast}}^{(q^\ast -2)/(q^\ast - 1)}
 \le C_q \, \| f \|_{L^1}^{1/(q^\ast - 1)} \, \| \nabla f \|_{L^{q}}^{(q^\ast -2)/(q^\ast - 1)}.
\end{equation}
Gathering these two inequalities, it comes
$$
\|\nabla f\|_{L^q} \le  C_q \,  I(f)^{1/2} \, \| f \|_{L^1}^{1/ (2 (q^\ast - 1))} \, 
\| \nabla f \|_{L^{q}}^{(q^\ast -2)/(2(q^\ast - 1))},
$$
from which we easily deduce \eqref{eq:nablaLqbdFisher} using that $f\in\PPP(\rr^2)$. 

\vip

We now verify (\ref{eq:LpbdFisher}).
For $p \in [1,\infty)$, write $p = q^\ast /2  =  q/(2-q)$ with  
$q:=2p/(1+p) \in [1,2)$ and use \eqref{eq:Lq*bdFisher} and \eqref{eq:nablaLqbdFisher}:
$$
||f||_{L^p} \leq C_p ||f||_{L^1}^{1/(q^\ast-1)}  
I(f)^{(3/2-1/q)(q^\ast-2)/(q^\ast-1)},
$$
from which one easily concludes since $f\in\PPP(\rr^2)$. 
\end{proof}

As a first consequence, we deduce that pairs of particles which law has finite Fisher information are not 
too close in the following sense.

\begin{lem} \label{lem:FishInteg}
Consider $F \in \PPP(\rr^2 \times \rr^2)$ with finite Fisher information and 
$(\XX_1,\XX_2)$ a random variable with law $F$. 
Then for any $\gamma \in (0,2)$ and any $\beta>\gamma/2$ there exists $C_{\gamma,\beta}$ so that 
\beqn\label{eq:f12Fisher}
\E (|\XX_1 - \XX_2|^{-\gamma}) = \int_{\R^2 \times \R^2}\frac{F(x_1,x_2)}{|x_1-x_2|^\gamma} \, dx_1dx_2 \le
C_{\gamma,\beta} \, (I(F)^{\beta} + 1).
\eeqn
\end{lem}

\begin{proof}  We introduce the unitary linear transformation
$$
\forall \, (x_1,x_2) \in \R^2 \quad \Phi  (x_1,x_2) =  \frac1{\sqrt 2}  
\bigl(x_1- x_2,x_1+x_2\bigr) =:   (y_1,y_2).
$$
Defining $\tilde F := F \circ \Phi^{-1}$ which is nothing but 
the law of $\frac1{\sqrt 2}  \bigl(\XX_1- \XX_2,\XX_1+\XX_2\bigr)$ and
$\tilde f$ as  the first marginal of $\tilde F$ (the law of 
$\frac1{\sqrt 2}  (\XX_1- \XX_2)$). A simple substitution shows that $I(\tilde F) = I(F)$.
Furthermore, the super-additivity property of Fisher's information proved in
\cite[Theorem 3]{Carlen91} (the factor $2$ below is due to our normalized definition 
of the Fisher information), see also 
\cite[Lemma 3.7]{hm}, implies that
\beqn\label{eq:FisherSuper}
I (\tilde f) \le 2 \,  I(\tilde F)=2 \,  I(F) .
\eeqn
Let $\beta \in (\gamma/2,1)$ be fixed 
(the case $\beta \geq 1$ will then follow immediately).
We have 
\bean
 \int_{\R^2 \times \R^2}\frac{F(x_1,x_2)}{|x_1-x_2|^\gamma} \, dx_1dx_2
 &=&  2^{\gamma/2} \int_{\R^2 \times \R^2}\frac{\tilde F(y_1,y_2)}{|y_1|^\gamma} \, dy_1dy_2
\\
&=& 2^{\gamma/2}  \int_{\R^2}\frac{\tilde f (y)}{|y|^\gamma} \, dy \\
&\le&   2^{\gamma/2}+  2^{\gamma/2} \, \int_{|y|\leq 1} \frac{\tilde f (y)}{|y|^\gamma} \, dy   . 
\eean
Using the H\"older inequality, that $\gamma / \beta <2$ and \eqref{eq:LpbdFisher}, we deduce that 
\bean
\int_{\R^2 \times \R^2}\frac{F(x_1,x_2)}{|x_1-x_2|^\gamma} \, dx_1dx_2
&\le& 2^{\gamma/2} + 2^{\gamma/2}  
\left[\int_{|y|\leq 1}|y|^{-\gamma /\beta}dx \right]^{\beta} \, \| \tilde f \|_{L^{1/(1-\beta)}(\rr^2)} \\
&\le & C_{\gamma,\beta}  ( 1 + I(\tilde f)^{\beta}).
\eean
We conclude thanks to \eqref{eq:FisherSuper}. 
\end{proof}

We also need something similar to Lemma \ref{lem:FishInteg1} that can be applied to vorticity measures.

\begin{lem} \label{lem:FishInteg2}
Consider a probability measure $g$ on $\rr\times \rr^2$ with $\supp g \subset [-A,A]\times \rr^2$
and define the probability measure $v$ and the (signed) measure $w$ on $\rr^2$ by 
$$v(B)=\int_{\rr\times B} g(dx,dm), \quad w(B)=\int_{\rr\times B} m g(dx,dm), \qquad \forall B\in\cB(\rr^2).$$
We have
\begin{align} 
\label{eq:LpbdFisher2} 
&\forall \; p\in [1,\infty), \quad  \| v \|_{L^p(\rr^2)} +
\| w \|_{L^p(\rr^2)} \leq C_{p,A} \tI(g)^{1 - 1/p},
\\
\label{eq:nablaLqbdFisher2}
&\forall \; q\in [1,2), \quad \| \nabla v \|_{L^q(\rr^2)} +
  \| \nabla w \|_{L^q(\rr^2)} 
\leq C_{q,A} \tI(g)^{ {3/2}-{1/q}}.
\end{align}
\end{lem}

\begin{proof}  We disintegrate $g(dm,dx)=r(dm)f^m(dx)$.
For $p \in [1,\infty)$, using the support condition on $g$, that the $L^p$-norm is convex and
(\ref{eq:LpbdFisher}) (since $f^m\in\PPP(\rr^2)$ for each $m \in \rr$), we get
\begin{align*}
|| w ||_{L^p}=\left\|\int_\rr m \, r(dm)f^m(x)\right\|_{L^p}
\leq A  \int_\rr r(dm) \|f^m\|_{L^p}
\leq A C_p  \int_\rr r(dm) I(f^m)^{1-1/p}.
\end{align*}
Since $r\in\PPP(\rr)$, the Jensen inequality leads us to
\begin{align*}
|| w ||_{L^p} \leq C_{p,A} \left(\int_\rr r(dm) I(f^m) \right)^{1-1/p} = C_{p,A} \tI(g)^{1-1/p}.
\end{align*}
The same proof works for $v$. Finally, \eqref{eq:nablaLqbdFisher2} is shown similarly, 
using \eqref{eq:nablaLqbdFisher} instead of \eqref{eq:LpbdFisher}.
\end{proof}

We end this section with some easy functional estimates.

\begin{lem}\label{efe}
Let $(w_t)_{t\geq 0} \in C([0,\infty),\MMM(\R^2))$ satisfy (\ref{cfish2}).
Then 
\beqn
\forall \;T>0, \quad \forall \; p\in (1,\infty), \quad w \in L^{p/(p-1)}([0,T], L^p(\R^2))
\eeqn
and
\beqn
\forall \;T>0, \quad \forall \; \gamma \in (0,2), \quad 
\int_0^T \int_{\R^2}\int_{\R^2} |x-y|^{-\gamma} |w_s(x)||w_s(y)| dydxds <\infty.
\eeqn
\end{lem}

\begin{proof}
The first estimate follows from (\ref{eq:Lq*bdFisher}) (applied with $q=2p/(1+p)$ and thus 
$q^\ast=2p$)
and the fact that $w \in L^\infty(0,T;L^1(\R^2))$
(because $w \in C([0,\infty),\MMM(\R^2))$).

To check the second estimate, consider $p>2/(2-\gamma)$ and observe that for any $x\in\R^2$,
by the H\"older inequality, since  $\gamma p/(p-1) < 2$,
$$
\int_{\{|x-y|\leq 1\}} |x-y|^{-\gamma}|w_s(y)|dy \leq \|w_s\|_{L^p}
\left( \int_{\{|x-y|\leq 1\}} |x-y|^{-\gamma p/(p-1)}dy\right)^{1-1/p} \leq C_{\gamma,p}  \|w_s\|_{L^p}.
$$
Consequently, 
\bean
\int_{\R^2}\int_{\R^2} |x-y|^{-\gamma} |w_s(x)||w_s(y)| dydx 
&\leq & \int_{\{|x-y|\geq 1\}}\int_{\R^2} |w_s(x)||w_s(y)| dydx \\
&&+ \int_{\{|x-y|< 1\}} |x-y|^{-\gamma} |w_s(x)||w_s(y)| dydx \\
&\leq & \|w_s\|_{L^1}^2  + C_{\gamma,p} \|w_s\|_{L^p} \|w_s\|_{L^1}.
\eean
We easily conclude using that  $w \in L^\infty(0,T;L^1(\R^2)) \cap  L^{p/(p-1)}([0,T], L^p(\R^2))$.
\end{proof}

\section{Many-particle entropy and Fisher information } \label{sec:prel2}
\setcounter{equation}{0}

We will need a result showing that if the particle distribution of the $N$-particle system has 
a uniformly bounded entropy and Fisher information, then any limit point of the associated empirical
measure has finite entropy and Fisher information.
As we will see, such a result is a consequence of representation identities for {\it level-3 functionals}
as first proven by Robinson and Ruelle in \cite{RR} for the entropy in a somewhat different setting. 
Recently in \cite{hm}, that kind of representation identity has been extended to the 
Fisher information. The proof is mainly based on the
De Finetti-Hewitt-Savage representation theorem~\cite{Hewitt-Savage,deFinetti1937}
(see also \cite{hm} and the references therein) together with 
convexity tricks for the entropy, and concentration for the Fisher
information. Unfortunately, we cannot apply directly the result of \cite{hm}
due to the additional variable corresponding to the circulations of vortices.
But the result still holds true and will be stated  in the next theorem after some
necessary definitions.

\vip

For a given $r \in \PPP(\R)$, we define $\Ee_N(r)$ as the set of probability measures 
$G^N \in \Psym( (\rr\times\rr^2)^N)$ such that $\int_{(\R^2)^N} 
G^N(dm_1,dx_1,\dots,dm_N,dx_N)= r^{\otimes N}(dm_1,\dots,dm_N)$.
We also denote by $\Ee_\infty(r)$ the set of probability measures $\pi \in \PPP(\PPP(\R \times \R^2))$ 
supported in $\{g \in \PPP(\R \times \R^2)\, : \, \int_{x\in\rr^2} g(dm,dx)=r(dm)\}$.
In addition, $\PPP_k(\PPP(\rr\times\rr^2))$ will denote the set of probabilities measure 
$\pi$ with finite moment 
$\tilde M_k( \pi) := \int_{\PPP(\rr\times\rr^2)} \tilde  M_k(g) \pi(dg) = \tMk(\pi_1)$, where
$\pi_1 := \int_{\PPP(\rr\times\rr^2)}g \ \pi(dg)\in\PPP(\rr\times\rr^2)$. 

\begin{thm}\label{th:levl3tH&tI} 
Let $k>0$ and $r\in \PPP(\R)$ supported in $\AA=[-A,A]$ for some $A>0$.
Consider, for each $N\geq 2$,  a probability measure $G^N \in \Ee_N(r)$. 
For $j\geq 1$, denote by $G^N_j\in \PPP((\R\times\R^2)^j)$ the $j$-th marginal of $G^N$.
Assume that $\sup_N \tMk(G^N_1) <\infty$ and that
there exists a compatible sequence $(\pi_j)$ of symmetric probability
measures on $(\rr\times\rr^2)^j$ so that $G^N_j \to \pi_j$ in the weak sense of measures in 
$\PPP((\rr\times\rr^2)^j)$. 
Denoting by $\pi \in \PPP_k(\PPP(\R \times \R^2))$ the probability measure associated to the 
sequence $(\pi_j)$ thanks
to the De Finetti-Hewitt-Savage theorem,  there holds 
\beqn\label{eq:level3H}
\int_{\PPP(\R \times \R^2)} \tilde H (g) \, \pi(dg) = \sup_{j \ge 1}Ê\tilde H(\pi_j) \leq \liminf_{N \to \infty} \tH (G^N),
\eeqn
as well as
\beqn\label{eq:level3I}
\int_{\PPP(\R \times \R^2)} \tilde I (g) \, \pi(dg)   = \sup_{j \ge 1}Ê\tilde I(\pi_j)
 \leq \liminf_{N \to \infty} \tI (G^N).
\eeqn
\end{thm}

The De Finetti-Hewitt-Savage theorem asserts that for a 
sequence $(\pi_j)$ of symmetric probability on $E^j$, compatible in the sense that 
the $k$-th marginal of $\pi_j$ is $\pi_k$ for all $1\leq k \leq j$, there exists a unique probability measure
$\pi \in \PPP(\PPP(E))$ such that $\pi_j= \int_{\PPP(E)} g^{\otimes j} \pi(dg)$. See for instance 
\cite[Theorem 5.1]{hm}.

\vip

Theorem~\ref{th:levl3tH&tI} is an immediate consequence of  \cite[Lemma 5.6]{hm}  and of the 
series of properties on the  partial entropy and Fisher information functionals that we state in the 
following lemma. 

\begin{lem}\label{lem:propHIJ} Let $k>0$ and $r\in \PPP(\R)$ supported in $\AA=[-A,A]$ for some $A>0$.
The partial entropy and Fisher information functionals
satisfy, with the common 
notation 
$\tJ = \tH$ or $\tI$, $\tJJ = \tHH$ or $\tII$, the
following properties

(i1) For any $j\geq 1$, $\tI : \PPP ( (\AA \times \R^2)^j) \to  \R \cup
\{+\infty \}$ is non-negative,
convex, proper and lower semi-continuous for the weak convergence.

(i2) For any $j\geq 1$, $\tH : \PPP_k( (\AA \times \R^2)^j) \to  \R \cup
\{+\infty \}$ is 
convex, proper, lower semi-continuous for the weak convergence and there exists some constant 
$C_{k} \in \R$ such that 
$$
 \PPP( (\AA \times \R^2)^j) \to  \R \cup
\{+\infty \}, \quad 
G \mapsto \tH (G) + \tilde M_{k/2}(G) + C_{k}
$$
is  lower semi-continuous for the weak convergence and non-negative.

(ii) For all $j\geq 1$, all $g \in \PPP_k(\AA \times \R^2)$, $\tJ (g^{\otimes j}) = \tJ (g)$.

(iii) For all $G \in \PPP_k( (\AA \times \R^2)^j)$, all $\ell,n$ with 
$j = \ell + n$ , there holds
$j \, \tJ (G) \ge \ell \, \tJ(G_\ell) + n \, \tJ(G_n)$, where $G_\ell\in \PPP_k( (\AA \times \R^2)^\ell)$ 
stands for the $\ell$-marginal of $G$.

(iv) 
The functional $\tJJ' : \PPP_k(\PPP(\AA\times\rr^2))\cap \Ee_\infty(r) \to \R \cup \{ \infty \}$ 
defined by 
$$
\tJJ' (\pi) := \sup_{j \ge 1} \tJ(\pi_j) \quad \hbox{where} \quad \pi_j:=\int_{\PPP(\AA\times\rr^2)}g^{\otimes j}
\pi(dg)
$$
is affine in the following sense. For any  $\pi \in
\PPP_k(\PPP(\AA\times\rr^2))$
and any partition of $\PPP_k(\AA\times\rr^2)$ by some sets $\omega_i$, $1\le i \le M$, such that
$\omega_i$ is an open set in $(\AA\times\rr^2) \backslash (\omega_1 \cup \ldots \cup
\omega_{i-1})$ for any $1\le i \le M-1$ and $\pi(\omega_i) > 0$ for any 
$1\le i \le M$, defining 
$$
\alpha_i := \pi(\omega_i) \quad \hbox{and} \quad \gamma^i := \frac{1}{\alpha_i} \, 
{\bf 1}_{\omega_i} \, \pi \in \PPP_k(\PPP(\AA\times\rr^2))
$$
so that 
$$
\pi = \alpha_1 \, \gamma^1 + ... + \alpha_M \, \gamma^M
\quad \hbox{and} \quad \alpha_1   + ... + \alpha_M =1,
$$
there  holds
$$
\JJ'(\pi) = \alpha_1 \, \JJ'(\gamma^1) + \ldots + \alpha_M \, \JJ'(\gamma^M).
$$
\end{lem}

 \begin{preuve} {\it of Lemma~\ref{lem:propHIJ}.} We only sketch the proof,
which is roughly an adaptation to the partial case of the proofs of~\cite[Lemma
5.5]{hm} and~\cite[Lemma 5.10]{hm}.

\vip

{\it Step 1. } We first prove point (i).

Let us first present alternative expressions of the 
entropy and the Fisher information.
For $G \in \PPP_k((\AA \times \R^2)^j)$, it holds that 
\bear\label{def1EntropB}
\tilde H(G) &=& 
\frac 1 j   \sup_{\phi \in C_c((\AA \times \R^2)^j )} \Bigl\langle G, \phi - H^*
(\phi ) + \log \theta^j \Bigr\rangle
\eear
with $\theta^j (M,X):= \theta^j(X) = c^j \, \exp (- |x_1|^k - ... - |x_j|^k)$, where 
$c$ chosen so that $\theta^j$ is a probability, where
$$
H^* (\phi) (M) :=  \int_{\R^{2j} }h^*(\phi(M,X)) \, \theta^j(X)  \, dX
$$
and where $h^*(t) := e^t - 1$ is the Legendre transform of $h(s) := s \, \log s - s + 1$.
The RHS term is well-defined in  $\R \cup \{+\infty \}$ because the function $\phi - H^*
(\phi ) + \log \theta^j$ is continuous and bounded by $- C \, \langle X \rangle^k$ by below for any 
$\phi \in C_c((\AA \times \R^2)^j )$. 

We also have, for  $G \in \PPP( (\AA \times \R^2)^j)$,
\beqn\label{def:Fisher2}
\tI (G)  :=  \frac1j\sup_{\psi   \in C_c^1( (\AA \times \R^2)^j)^{2j}} 
\Bigl  \langle G,  - \frac{| \psi |^2 } 4  - \hbox{div}_X\, \psi \Bigr\rangle.
\eeqn
Again, the RHS term is well defined in  $\R$ because the function $| \psi |^2 / 4  - \hbox{div}_X\, 
\psi $ is continuous and
bounded for any $\psi   \in C_c^1( (\AA \times \R^2)^j)^{2j}$. 

\vip
As a consequence of the representation formulas \eqref{def1EntropB} and \eqref{def:Fisher2} we immediately  
conclude that $\tilde H$ and $\tilde I$ are convex, lower semi-continuous 
and proper so that point (i1) and the first part of (i2) hold. The lower bound expressed in (i2) 
is nothing but 
the result stated in Lemma~\ref{ieth}. 

\vip

{\it Step 2.} Point (ii) is obvious from (\ref{dfth})-(\ref{dfti}).

\vip

{\it Step 3. } We now prove (iii). For the partial 
entropy, we define
$$
\tilde h_i := i \,  \tilde H(G_i)
$$
for any $G \in \PPP_k((\AA \times \R^2)^j)$ and $1 \le i \le j$, and we just
write $G_i = R_i \, F_i$ 
instead of 
the more explicit expression $G_i(dM,dX) = R_i(dM) \, F_i^M(X) \, dX$. We then
compute
\bean
\tilde h_j - \tilde h_i - \tilde h_{j-i} 
&=& \int_{(\AA\times \R^2)^j} G_j \log F_j -  \int_{(\AA\times \R^2)^j} G_i \log F_i   - 
\int_{(\AA\times \R^2)^{j-i}} \!\!\! G_{j-i} \log F_{j-i} 
\\
&=& \int_{(\AA\times \R^2)^j} G_j  \, [\log F_j  -   \log F_i \otimes  F_{j-i} ] 
\\
&=& \int_{\AA^j} R_j \int_{(\R^2)^{j}} F_j   \, [\log F_j  -   \log F_i \otimes 
F_{j-i} ]
\\
&=& \int_{\AA^j} R_j \int_{(\R^2)^j}F_i \otimes  F_{j-i} \, [ u \, \log u - u +1
] \ge 0,
\eean
where we have set $u := F_j/(F_i \otimes  F_{j-i})$ and we have used that $F_j, 
F_i \otimes  F_{j-i} \in 
\PPP((\R^2)^{j})$
for any given $M \in \R^j$. 

\vip
For the partial Fisher information we reproduce the proof of the same
super-additivity property established
for the usual Fisher information in \cite[Lemma~3.7]{hm}. We define for any $i \le j$
$$
\tilde \iota_i  := i \, \tilde I (G_i) =    \sup_{\psi   \in C_c^1( (\AA \times \R^2)^i)^{2i}} 
\int_{(\AA \times \R^2)^i }   \Bigl( \nabla_X \, G_i \cdot \psi  -
G_i \, \frac{|\psi|^2}4  \Bigr)
$$
where the sup is taken on all $\psi =(\psi_1,\ldots,\psi_i)$, with all 
$\psi_\ell :  (\AA \times \R^2)^i \to \R^2$. 
We then write the previous equality for $\iota_j$ and restrict the supremum over all $\psi$ such that for some $i \le j$:

$\bullet$ the $i$ first $\psi_\ell$ depend only on $(x_1,\ldots,x_i)$, with the
notation $\psi^i = (\psi_1, \ldots,\psi_i)$,

$\bullet$ the $j-i$ last $\psi_\ell$ depend only on $(x_{i+1},\ldots,x_j)$,
with the notation $\psi^{j-i} = (\psi_{i+1}, \ldots,\psi_j)$.

\noindent
We then have the inequality
\bean
\tilde \iota_j& \geq &  
\sup_{\psi^i ,\, \psi^{j-i} } 
\int_{(\AA \times \R^2)^j }   [\nabla_i G  \cdot \psi^i + \nabla_{j-i} G \cdot
\psi^{j-i} - G \, \frac{|\psi^i|^2+ |\psi^{j-i}|^2}4 ] \\
& = &  \sup_{\psi^i   \in C_c^1( (\AA \times \R^2)^i)^{2i}} 
\int_{(\AA \times \R^2)^i }   [\nabla_i G_i  \cdot \psi^i  - G_i \, \frac{|\psi^i|^2}4 ]
\\ 
&& \hspace{2cm} + 
\sup_{\psi^{j-i}   \in C_c^1( (\AA \times \R^2)^{j-i})^{2(j-i)}} 
\int_{(\AA \times \R^2)^{j-i} }   [\nabla_{j-i} G_{j-i}  \cdot \psi^{j-i}  - G_{j-i} \, \frac{|\psi^{j-i}|^2}4 ]
\\
& = & \tilde \iota_i + \tilde \iota_{j-i},
\eean
where all the gradients appearing are only gradients in the $X$ variables.

\vip

{\it Step 3. } We first note that as a consequence of (iii) we have (see \cite[Lemma 5.5]{hm} 
for details) for any 
$ \pi \in  \PPP_k(\PPP(\AA \times \R^2))$
\beqn\label{eq:JJprime=limit}
\tilde \JJ'(\pi) := \sup_{j \ge 1} \tilde J(\pi_j) = \lim_{j \to \infty}\tilde J(\pi_j).
\eeqn
We now prove the affine caracter (iv) for the partial entropy  $\tilde \HH'$,
considering only the case $M=2$ for simplicity.
Let us consider $A, B \in \PPP_k(\PPP(\AA \times \R^2))\cap \Ee_\infty(r)$, $\theta \in (0,1)$, 
and let us introduce the disintegration   $A_j = R_j \alpha_j$, $B_j = R_j \beta_j$,
with $R_j=r^{\otimes j}$ (because both $A$ and $B$ belong to $\Ee_\infty(r)$).
Using that  $s \mapsto \log s$ is an increasing function and that $s \mapsto s \, \log s$ is a convex 
function, we have 
\bean
\tilde H (\theta \, A_j + (1-\theta) \, B_j) 
&=&  \frac1j \int_{(\AA \times \R^2)^j} R_j \,  (\theta \,  \alpha_j + (1-\theta) \, \beta_j)  \, 
\log (\theta \,  \alpha_j + (1-\theta) \, \beta_j) \\
&\ge& \frac1j \int_{(\AA \times \R^2)^j} R_j \, \{\theta \,  \alpha_j \, \log (\theta \,  \alpha_j ) + 
(1-\theta) \, 
\beta_j  \, \log ( (1-\theta) \, \beta_j) \}\\
&=& \theta \, \tilde H ( A_j) + (1-\theta) \, \tilde  H(B_j) + \frac1j [ \theta  \, \log \theta + 
(1-\theta)  \, \log (1-\theta)] \\
&\ge& \tilde  H( \theta \,   A_j +  (1-\theta) \, B_j) +  \frac1j [ \theta  \, \log \theta + 
(1-\theta)  \, \log (1-\theta)].
\eean
Passing to the limit $j\to\infty$ in the two preceding inequalities and using  \eqref{eq:JJprime=limit}, 
we get 
$$
\tilde \HH' (\theta \,  A + (1-\theta) \, B) \ge \theta \, \tilde \HH' ( A) + 
(1-\theta) \, \tilde \HH' (B) \ge \tilde \HH' (\theta \,  A + (1-\theta) \, B),
$$
from which the announced affine caracter follows.

\vip

We next prove the affine caracter (iv) for the partial Fisher information. 
For the sake of simplicity we only consider the case when $M=2$ and $\omega_1$
is a ball. The case when 
$\omega_1$ is a general open set can be handled in a similar way and the case
when $M \ge 3$ can be deduced 
by an iterative argument. For some given $\pi \in \PPP_k(\PPP(\rr \times
\rr^2))\cap \Ee_\infty(r)$ which is not a Dirac mass, some
$f_1 \in \PPP_k(\rr \times \rr^2)$ and some $\eta \in (0,\infty)$
so that 
$$
\theta := \pi(B_\eta) \in (0,1), \quad B_\eta = B(f_1,\eta) := \{ \rho, W_1(\rho,f_1) < \eta\},
$$
we define 
$$
A := \frac1 \theta \indiq_{B_\eta} \pi, \quad B := \frac1{1-\theta} \indiq_{B_\eta^c} \pi
$$
so that 
$$
A, B \in \PPP_k(\PPP(\AA\times\rr^2)) \cap \Ee_\infty(r)
\quad\hbox{and} \quad \pi = \theta A+ (1- \theta )B,
$$
and we have to prove that 
\beqn\label{eq:linearIIprim}
\tilde \II'(\pi) = \theta \, \tilde \II'(A) + (1-\theta) \, \tilde \II'(B).
\eeqn
We claim that proceeding as in the proof of \cite[Lemma 5.10]{hm} we may assume, up to regularization by convolution in the $X$ variables, that 
$$
\sup_{j,M,X} \Bigl(  |\nabla_1 \log \pi_j| +  |\nabla_1 \log A_j| +  |\nabla_1 \log B_j| \Bigr)  \le C <  \infty,
$$
where the $\nabla_1$ stands for the gradient in the first position variable only.
For any given $j \ge 1$, we   define
$$
Z_{j} :=  \theta \, \tilde I (A_{j}) + (1-\theta) \, \tilde I(A_{j}) 
- \tilde I(\theta \, A_{j} + (1-\theta) \, B_{j}),
$$
and after some calculations, we obtain, using the disintegrations $A_j=R_j \alpha_j$ and
$B_j= R_j \beta_j$ as previously,
\bean
Z_{j} &=&  \, \theta (1-\theta) \int R_j \, \frac{\alpha_{j}
\beta_{j}}{(1-\theta)\alpha_{j} + \theta \beta_{j}}
\biggl| \nabla_1 \log \frac{\alpha_{j}}{\beta_{j}}\biggr|^2 \\
&\le &
 2 \, \theta (1-\theta) \int R_j \, \frac{\alpha_{j}
\beta_{j}}{(1-\theta)\alpha_{j} + \theta \beta_{j}} 
\left(\bigl| \nabla_1 \log \beta_{j} \bigr|^2+ \bigl| \nabla_1 \log \alpha_{j}\bigr|^2 \right)
 \\
&\le& 
4
   \theta (1-\theta)  \, C \int R_j \, \frac{\alpha_{j}
\beta_{j}}{(1-\theta)\alpha_{j} + \theta \beta_{j}}   
= 4
   \theta (1-\theta)  \, C \int  \, \frac{A_j
B_j}{(1-\theta)A_j + \theta B_j}.
\eean
At this stage, the proof follows exactly the one done in~\cite[Lemma 5.10]{hm} which states that  
the same property holds for the full Fisher information. Let us introduce, for any $s \in (0,\eta)$ the two measures on $\PPP(\AA\times\rr^2)$ (which  are not necessarily probability measures)
$$
A' :=  \indiq_{B_s} A  =\frac1 {\theta} \indiq_{B_s} \pi,\qquad 
A'' := \indiq_{B_\eta\backslash B_s} A,
$$
and let us observe that 
$$
\lim_{s \to \eta } \int A''(d\rho) = 0  ,
$$
by Lebesgue's dominated convergence theorem.  
By construction, there holds $A' + A'' = A$ as well as for any   $j \ge 1$  there holds $A_{j} '+ A_{j}''=A_{j}$
with $A''_{j}\ge 0$, so that we may write for any $\eps > 0$ 
\bean
Z_{j} &\le & 
4
   \theta (1-\theta) \, C \int_{\PPP(\AA\times\rr^2)}  \frac{B_j
A'_j}{(1-\theta) B_j + \theta A'_j}   +  \eps
\eean
taking $s$ close enough to $r$ (independently of $j$).  We introduce the notation $y=(m,x)$ for 
the couple circulation-position,  define the distance $d(y,y') := \min(|y-y'|,1)$ on 
$\rr \times \rr^2$ and the Monge-Kantorovitch-Wasserstein distance $W_1$ defined on 
$\PPP(\rr \times \rr^2)$ according to distance $d$. 

We introduce the real numbers $u= \frac{\eta+s}2$ and  $\delta = \frac{\eta-s}2$, 
as well as the set 
$$
\tilde B_{u} := \{ Y^j=(y_1,\ldots,y_j) \;, \;  W_1(\mu^j_{Y^j}, f_1) < u \}
\subset (\rr\times\rr^2)^j
$$
which is nothing but the reciprocal image of the ball $B_u \subset \PPP(\rr\times\rr^2)$ 
by the empirical measure map. Using that  
$$
\frac{B_j A'_j}{(1-\theta) B_j + \theta A'_j} \le \frac1\theta B_j \, {\bf
1}_{\tilde B_u} + 
\frac1{ 1-\theta} A'_j  \, {\bf1}_{\tilde B_u^c}, 
$$
we get 
\beqn\label{eq:Atj}
Z_{j} \leq 
4\, C \left( 
 (1-\theta) \int_{\tilde B_{u}}  B_{j} +  \theta \int_{\tilde B_{u}^c} A'_{j}  \right)
 +  \eps.
\eeqn
Using concentration of empirical measures exactly as in Step 3 of the proof of \cite[Lemma 5.11]{hm}
we deduce that 
$$
Z_j   \le 
\frac{ 4 \, C  \, [A^k + \tMk(\pi)]^{1/k} }{\delta j^\gamma} +  \eps,
$$
 with $\gamma := 1/(5+3/k)$.  Remark that we have the bound $A^k + \tMk(\pi)$ 
for the full moment in $z=(m,x)$ of the probability $\pi$, since the $m$ variable 
always belongs to $\AA=[-A,A]$.
Using the above estimate and the convexity estimate 
$Z_j \ge 0$, we 
obtain that 
$$
\lim_{j \to 0} Z_j = 0
$$
from which we conclude using \eqref{eq:JJprime=limit}.
\end{preuve}

As a consequence of Lemma~\ref{lem:propHIJ}, we also have 
some super-additivity inequalities as well as some weak lower semi-continuity properties
that we  will frequently use.

\begin{cor}\label{cor:fishdec} 
Let $k>0$ and $r\in\PPP(\rr)$ supported in $\AA=[-A,A]$ for some $A>0$.

(i) Let $G \in \Ee_N(r)$ for some $N\geq 2$. For any $ 1\leq  j \leq N$, 
denoting by $G_j$ the $j$-marginal of $G$ and 
introducing the Euclidian decomposition $N = n \, j + \ell$, $0 \le \ell \le j-1$,
there holds
\beqn\label{eq:supaddF&H}
\tI(G_j)\leq  (1+ \frac{\ell}{nj} ) \,  \tI(G) \le 2 \,   \tI(G) 
\quad\hbox{and}\quad
\tH(G_j)\leq  (1+ \frac{\ell}{nj} ) \,  \tH(G) + \frac{\ell}{nj} \, (C_{k} +  \tilde M_{k/2}(G_\ell) ). 
\eeqn

(ii) Let $j\geq 1$ be fixed and $\pi_j \in \Ee_j(r)$. Consider a sequence $G^N \in \Ee_N(r)$
such that $G^N_j \wto \pi_j$ weakly in $\PPP((\R \times \R^2)^j)$
as $N \to \infty$ and $\sup_N \tilde M_k(G^N_1)<\infty$. Then
$$
\tH (\pi_j) \leq \liminf_{N \to \infty} \tH (G^N)
\quad\hbox{and}\quad
\tI (\pi_j) \leq \liminf_{N \to \infty} \tI (G^N).
$$
\end{cor}

\begin{preuve}  {\it of Corollary~\ref{cor:fishdec}.} We start with point (i).
Iterating the super-additivity property expressed in Lemma~\ref{lem:propHIJ}-(iii) tells us that
\beqn\label{eq:superaddJ}
 \ell \tilde J(G_\ell)+nj \tilde J(G_j) \le N \tilde J(G),
\eeqn
for $\tilde J = \tH$ and $\tilde J = \tI$.  In the case of the Fisher information (which is non-negative), 
we deduce that 
$nj \tilde I(G_j) \le N \tilde I(G) $ which implies the  first assertion 
in \eqref{eq:supaddF&H}. 
For the entropy, \eqref{eq:superaddJ} together with the non-negativity property
$\ell \,  \tH (G_\ell) + \ell \,  (\tilde M_{k/2}(G_\ell) + C_{k})  \ge 0$ established in 
Lemma~\ref{lem:propHIJ}-(i2)
 imply the last  assertion in \eqref{eq:supaddF&H}. 
 
\vip
We next check (ii).
The lower semi-continuity stated in Lemma~\ref{lem:propHIJ}-(i) implies that
$$
\tilde J( \pi_j) \le \liminf_{N \to \infty}  \tilde J(G^N_j)
$$
for $\tilde J = \tH$ and $\tilde J = \tI$. We conclude using point (i).
\end{preuve}

\section{Main estimates and tightness}\label{sec:smainest}
\setcounter{equation}{0}

In the whole section, a family $G^N_0 \in \Psym((\rr\times\rr^2)^N)$ satisfying \eqref{chaosini}
for some $k\in (0,1]$, some $A \in (0,\infty)$ and some $g_0 \in \PPP(\rr\times\rr^2)$ is fixed. 
The following estimate is central in our proof.

\begin{prop}\label{cutoff}
For $N\geq 2$, let $(\MM_i^N,\XX_i^N(0))_{i=1,\dots,N}$ be
$G^N_0$-distributed and consider the unique solution $(\XX^N_i(t))_{i=1,\dots,N,t\geq 0}$ to (\ref{ps}). 
For $t\geq 0$, denote by $G^N_t \in \Psym((\rr\times\rr^2)^N)$ the law of $(\MM_i^N,\XX_i^N(t))_{i=1,\dots,N}$.
There is a constant $C = C(\sigma,k,A)$ such that for all $t\geq 0$,
\bear
\tilde H(G^N_t) + \tilde M_k(G^N_t) + \frac \nu2 \int_0^t\tilde I (G^N_s)ds \leq
\tilde H(G^N_0) + 
\tilde M_k(G^N_0) + C \, t.
\label{eq:entrop}
\eear
As a consequence, there exists a constant $C$ which depends 
furthermore on an upper bound on $ \tilde H(G^N_0)$ and $ \tilde M_k(G^N_0)$ so that for all $N \geq 2$, 
all $t\geq 0$,
\bear
\tilde H(G^N_t) \leq C(1+t), \quad 
\tilde M_k(G^N_t) \leq C(1+t) \quad \hbox{and} \quad
  \int_0^t\tilde I (G^N_s)ds \leq C(1+t).\label{eq:aprioribds}
\eear
\end{prop}

\begin{proof}
The computations below are formal.
To handle a rigorous proof, it suffices to approximate the singular
kernel $K$ by a smoothed kernel $K_\e$ enjoying the properties that div $K_\e=0$ and that $K_\e(x)=K(x)$
for all $|x|\geq \e$, which makes all
the computations below rigorous. 
Since (\ref{ps}) is well-posed thanks to Osada \cite{o2} (see Theorem \ref{wpps}) and since 
the functionals $\tilde M_k$, $\tilde H$ and $\tilde I$ are lower semi-continuous 
for the weak convergence, it is not hard to conclude to \eqref{eq:entrop} (with only an inequality now).

\vip
{\it Step 1.}  Denoting by  $X=(x_1,\dots,x_N)$ and $M=(m_1,\dots,m_N)$, we disintegrate $G^N_t(dM,dX)$ as 
$R_t^N(dM)F^{N,M}_t(dX)$ 
and we observe that $F^{N,M}_t$ is nothing but the conditional law of 
$(\XX^{N}_i(t))_{i=1,\dots,N}$ knowing that $(\MM_i^N)_{i=1,\dots,N}=M$. 
We also observe that $R_t^N(dM) = R^N_0(dM)$, because the circulations 
$\MM_i^N$ do 
not depend on  time. 
Conditionally on $(\MM_i^N)_{i=1,\dots,N}=M$, $(\XX^{N}_1(t),\dots,\XX^{N}_N(t))$ solves 
\begin{align}\label{pscond}
\forall i=1,\ldots,N,\quad \XX_i^{N}(t) = \XX_i(0)+  \frac1N \sum_{j \neq i} \intot m_j K(\XX_i^{N}(s) -
\XX_j^{N}(s))ds +  \sigma \BB_i(t).
\end{align}
Applying the It\^o formula to compute the conditional expectation of 
$\varphi(\XX^{N}_1(t),\dots,\XX^{N}_N(t))$ knowing that $(\MM_i^N)_{i=1,\dots,N}=M$, we get, for any $\varphi \in
C^2_b((\rr^2)^N)$, any $t\geq 0$,
\begin{align}\label{eqwfn}
\frac d {dt} \int_{(\rr^2)^N} \varphi(X) F^{N,M}_t(dX) =& 
\int_{(\rr^2)^N} \left[\frac1N \sum_{i\ne j} m_j K(x_i-x_j)\cdot \nabla_{x_i}\varphi(X)\right]F^{N,M}_t(dX)\\
&+ \nu  \int_{(\rr^2)^N} \Delta_{X}\varphi(X). \nonumber
\end{align}
Consequently (recall that div $K=0$), $F^{N,M}$ is a weak solution to
\begin{align}\label{eqfo}
\partial_t F^{N,M}_t(X) 
+ \frac1N \sum_{i\ne j} m_j K(x_i-x_j) \cdot 
\nabla_{x_i}F^{N,M}_t(X) = \nu \Delta_{X} F^{N,M}_t(X).
\end{align}

{\it Step 2.} We easily compute the evolution of the entropy (in the space variable):
\bean
\frac d { dt}H(F^{N,M}_t) 
&=& \frac 1N \int_{(\rr^2)^N} (\partial_t F^{N,M}_t(X)) (1+\log (F^{N,M}_t(X))) \, dX\\
&=& -\frac 1{N^2} \sum_{i\ne j} m_j \int_{(\rr^2)^N}  K(x_i-x_j). \nabla_{x_i} F^{N,M}_t(X) (1+\log (F^{N,M}_t(X)))
\, dX \\
&&+  \frac \nu N \int_{(\rr^2)^N}  \Delta_{X} F^{N,M}_t(X) \,  (1+\log
(F^{N,M}_t(X))) \, dX.
\eean
Observing that the first term vanishes (because div $K=0$) 
and performing an integration by parts on the second term,
we immediately and  classically deduce that
\begin{align}\label{entropFNm}
H(F^{N,M}_t) +  \nu \intot I(F^{N,M}_s) \, ds = H(F^{N,M}_0), \quad \forall \, t
\ge 0. 
\end{align}
Integrating this equality against $R_0(dM)$, we finally get
\begin{align}\label{entropGN}
\tilde H(G^{N}_t) +  \nu \intot \tilde I(G^{N}_s) ds = \tilde H(G^{N}_0), \quad
\forall \, t \ge 0. 
\end{align}

{\it Step 3.} Applying \eqref{eqwfn} with $\varphi(X)= \langle X \rangle^k$ 
(for which $|\nabla_{x_i}\varphi| \leq C/N$ and $|\Delta_X \varphi|\leq C$ because
$k \in (0,1]$) and integrating against $R_0^N(dM)$, we get
\bean
\frac d { dt} \tMk(G^N_t)
&\leq&\frac C {N^2} \int_{\rr^N} R_0^N(dM) \int_{(\rr^2)^N} F^{N,M}_t(dX) 
\sum_{i\ne j}  |m_j| | K(x_i-x_j) | \\
&&+  C  \int_{\rr^N} R_0^N(dM) \int_{(\rr^2)^N} F^{N,M}_t(dX)\\
&\le& C A  \int_{(\rr\times \rr^2)^N} G^N_t(dM,dX) |K(x_1-x_2)|  +   C.
\eean
For the last inequality, we used that $R_0^N(dM)F^{N,M}_t(dX)=G^N_t(dM,dX)$
is a symmetric probability measure supported in $([-A,A]\times\rr^2)^N$.
Denoting by $G^{N}_{t2}$ the two-marginal of $G^{N}_t$, disintegrating 
$G^{N}_{t2}(dm_1,dx_1,dm_2,dx_2)=r_t^N(dm_1,dm_2)f_t^{N,m_1,m_2}(dx_1,dx_2)$
and using Lemma
\ref{lem:FishInteg} with $\gamma = 1$ and $\beta= 2/3$, then the Jensen inequality 
($r^N_t$ is a probability measure) and finally the definition of $\tI$, we find 
\bean
\int_{(\R\times\R^2)^N}   |K(x_1-x_2)| \,  G^{N}_t(dM,dX) 
&=& \int_{(\rr\times \R^2)^2} \frac1{|x_1-x_2|} G^{N}_{t2} (dm_1,dm_2,dx_1,dx_2)\\
&=& \int_{\R^2} r_t^N(dm_1,dm_2) \int_{(\rr^2)^2} \frac1{|x_1-x_2|}f_t^{N,m_1,m_2}(dx_1,dx_2)\\
&\leq & C \int_{\R^2} r_t^N(dm_1,dm_2) \left( 1+ I(f_t^{N,m_1,m_2})^{2/3}\right)\\
&\leq & C + C\left(\int_{\R^2} r_t^N(dm_1,dm_2) I(f_t^{N,m_1,m_2})\right)^{2/3}\\
&\le& C + C \tI(G^{N}_{t2})^{2/3}.
\eean
 Finally, using Corollary \ref{cor:fishdec}, we have
$\tI(G^{N}_{t2}) \leq 2 \tI(G^N_t)$,  so that
$$
\frac d { dt}\tMk(G^N_t)  
\le C + C  \tilde I ( G^{N}_{t} )^{2/3} \leq C +  \frac\nu2 \tilde I (
G^{N}_{t}) .
$$
For the last inequality, we recall that the value of $C$ is allowed to change and we mention that we used 
the inequality $Cx^{2/3} \leq C' + (\nu/2)x$ for all $x\geq 0$.
Integrating in time, we thus get
\beqn\label{MkGN}
\tMk(G^N_t)  
\le Ct + \tMk(G^N_0) + \frac\nu 2  \intot \tI ( G^{N}_{s} )ds.
\eeqn
{\it Step 4.}
Summing \eqref{entropGN} and \eqref{MkGN},  we thus find \eqref{eq:entrop}.
This implies the first inequality in  (\ref{eq:aprioribds}) by positivity of $\tMk$ and $\tI$. 
Finally, we write 
$$
 \tMk(G^N_t) + \frac\nu 2 \intot \tI ( G^{N}_{s} )ds \leq C(1+t) -\tH(G^N_t)
\leq C(1+t) +\tMk(G^N_t)/2
$$
by Lemma \ref{ieth} (with the choice $\lambda=1/2$). 
Thus $\tMk(G^N_t) + \nu \intot \tI ( G^{N}_{s} )ds \leq C(1+t)$ 
which implies the second and third  inequalities in (\ref{eq:aprioribds}) by positivity of $\tMk$ and 
$\tI$ again.
\end{proof}

We can now easily prove the tightness of our particle system.

\begin{lem}\label{tight}
For each $N\geq 2$, recall that $(\MM_i^N,\XX_i^N(0))_{i=1,\dots,N}$ is $G^N_0$-distributed
and consider the unique solution
$(\XX^{N}_i(t))_{i=1,\dots,N,t\geq 0}$ to (\ref{ps}) .
We also set $\QQ^{N}:=N^{-1}\sum_{i=1}^N\delta_{(\MM_i^N,(\XX^{N}_i(t))_{t\geq 0})}$.

(i) The family $\{\LL((\XX^{N}_1(t))_{t\geq 0}), N\geq 2\}$ is tight in $\PPP(C([0,\infty),\rr^2))$.

(ii) The family $\{\LL(\QQ^{N}), N\geq 2\}$ is tight in $\PPP(\PPP(\rr\times C([0,\infty),\rr^2)))$.
\end{lem}

\begin{proof}
First, point (ii) follows from
point (i). Indeed, $\MM_1^N$ takes values in the compact set $[-A,A]$, so that we deduce from
(i) that the family $\{\LL(\MM_1^N,(\XX^{N}_1(t))_{t\geq 0}), N\geq 2\}$ 
is tight in $\PPP(\rr\times C([0,\infty),\rr^2))$. Then (ii) follows from the exchangeability 
of the system, see \cite[Proposition~2.2]{SSF} or  \cite[Lemma 4.5]{m}.

\vip

To prove (i), we have to check that for all $\eta>0$ and all $T>0$, we can find a compact
subset $\cK_{\eta,T}$ of $C([0,T],\rr^2)$ such that $\sup_{N} \Prob[(\XX^{N}_1(t))_{t\in [0,T]} \notin
\cK_{\eta,T}]\leq \eta$. Let thus $\eta>0$ be fixed.
We introduce the random variable $Z_T:=\sup_{0<s<t<T}|\sigma \BB_1(t)-\sigma \BB_1(s)|/|t-s|^{1/3}$ 
which is a.s. finite,
since the paths of $\BB_1$ are a.s. H\"older continuous with index $1/3$. Note also
that the law of $Z_T$ does not depend on $N$. Next, we use the H\"older inequality 
and the fact that a.s., $|\MM_i^N|\leq A$ for all $i$ (recall (\ref{chaosini})) to get,
for all $0<s<t<T$,
\begin{align*}
\Big|\frac1N\int_s^t \sum_{j\ne 1}\MM_j^N K(\XX^{N}_1(u)-\XX^{N}_j(u))du\Big|\leq&
\frac A N \int_s^t \sum_{j\ne 1}|\XX^{N}_1(u)-\XX^{N}_j(u)|^{-1} du\\
\leq& \frac A N (t-s)^{1/3}
\sum_{j\ne 1}\left[\int_0^T |\XX^{N}_1(u)-\XX^{N}_j(u)|^{-3/2} du\right]^{2/3}\\
\leq&  (t-s)^{1/3}\left[ A + 
\frac A N\sum_{j\ne 1}\int_0^T |\XX^{N}_1(u)-\XX^{N}_j(u)|^{-3/2} du \right]\\
=:& U_T^{N} (t-s)^{1/3}.
\end{align*}
All this yields that for all  $0<s<t<T$ (recall that $\XX^N_1$ satisfies the first equation of
(\ref{ps})), 
$$
|\XX^{N}_1(t)- \XX^{N}_1(s)|\leq (Z_T+U_T^N) (t-s)^{1/3}.
$$
By exchangeability and using Lemma \ref{lem:FishInteg}, 
$$
\E[U_T^{N}]= A + A\frac{N-1}{N}\int_0^T \E[|\XX^{N}_1(u)-\XX^{N}_2(u)|^{-3/2}] du
\leq A+A \int_0^T \tI(G^N_{u2})du,
$$ 
where $G^N_{u2}$ is the two-marginal of $G^N_u$. 
But  $\tI(G^N_{u2}) \leq 2 \tI(G^N_u)$ (by Corollary~\ref{cor:fishdec}),  so that using finally 
Proposition \ref{cutoff}, $\E[U_T^{N}]\leq A + CA(1+T)$.

\vip

Thus $\sup_{N\geq 2}\E[U_T^{N}]<\infty$ and since $Z_T$ is a.s. finite, we can clearly find $R>0$ such that 
$\Prob[Z_T+U_T^{N}> R]\leq \eta/2$ for all $N\geq 2$. 
We also know by (\ref{chaosini}) that 
$\sup_{N\geq 2} \E[\langle\XX^N_1(0)\rangle^k]=\sup_{N\geq 2} \tMk(G^N_0)<\infty$, 
so that there is $a>0$ such that $\sup_{N\geq 2} \Prob[|\XX_1^N(0)|>a]\leq \eta/2$.
Let now $\cK_{\eta,T}$ be the set of all continuous functions $f:[0,T]\mapsto\rr^2$
with $|f(0)|\leq a$ and $|f(t)-f(s)|\leq R (t-s)^{1/3}$ for all $0<s<t<T$.
For all $N\geq 2$, we have
$\Prob[(\XX^{N}_1(t))_{t\in[0,T]} \notin \cK_{\eta,T}]\leq \Prob[|\XX_1^N(0)|>a ]+ \Prob[ Z_T+U_T^{N}> R ]\leq \eta$.
Since $\cK_{\eta,T}$ is a compact subset of $C([0,T],\rr^2)$, this ends the proof.
\end{proof}

\section{Consistency}\label{sec:pr}
\setcounter{equation}{0}

In the whole section, we assume (\ref{chaosini}) for some $k\in(0,1]$, some $A>0$ and some 
$g_0\in \PPP(\rr\times\rr^2)$.
We define $\cS$ as the set of all probability measures 
$g\in\PPP(\rr\times C([0,\infty),\rr^2))$ such that $g$ is the law of
$(\MM,(\XX(t))_{t\geq 0})$ with $(\XX(t))_{t\geq 0}$ solution to the nonlinear SDE \eqref{NSDE} associated 
with $g_0$ 
and satisfying (\ref{tif}): for $g_t\in \PPP(\rr\times \rr^2)$ the law of $(\MM,\XX(t))$,
\begin{align*}
\forall \; T>0,\quad \int_0^T \tI(g_s)ds <\infty.
\end{align*}

\begin{prop}\label{consist}
For each $N\geq 2$, let $(\MM_i^N,\XX_i^N(0))_{i=1,\dots,N}$ be $G_0^N$-distributed and
consider the unique solution
$(\XX^{N}_i(t))_{i=1,\dots,N,t\geq 0}$ to (\ref{ps}). Assume that there is a subsequence
of $\QQ^{N}:=N^{-1}\sum_{i=1}^N\delta_{(\MM_i^N,(\XX^{N}_i(t))_{t\geq 0}}$ going in law to some 
$\PPP(\rr\times C([0,\infty),\rr^2))$-valued random variable $\QQ$. Then $\QQ$ a.s. belongs to $\cS$.
\end{prop}

\begin{proof}
We consider a (not relabelled) subsequence of $\QQ^N$ going in law to some $\QQ$.
We adopt in this proof the convention that $K(0)=0$.
\vip

{\it Step 1.} Consider the identity maps $m : \rr \mapsto \rr$ and $\gamma: C([0,\infty),\rr^2) 
\mapsto C([0,\infty),\rr^2)$.
Using the classical theory of martingale problems, we realize that $g$ belongs to $\cS$ as soon as

(a) $g\circ (m,\gamma(0))^{-1}=g_0$;

(b) setting $g_t= g\circ (m,\gamma(t))^{-1}$, (\ref{tif}) holds true; 

(c) for all $0<t_1<\dots <t_k<s<t$, all $\psi \in C_b(\rr)$, all
$\varphi_1,\dots,\varphi_k \in C_b(\rr^2)$, all $\varphi \in C^2_b(\rr^2)$,
\begin{align*}
\cF(g):=&\int\!\!\int g(dm,d\gamma)g(d\tm,d\tg)\psi(m)\varphi_1(\gamma_{t_1})\dots\varphi_k(\gamma_{t_k})\\
& \hskip0.8cm \left[
\varphi(\gamma_t)-\varphi(\gamma_s)-\int_s^t \tm K(\gamma_u-\tg_u)\cdot \nabla \varphi(\gamma_u)du - 
\nu 
\int_s^t \Delta\varphi(\gamma_u)du\right]=0. 
\end{align*} 

Indeed, let $(\MM,(\XX(t))_{t\geq 0})$ be $g$-distributed. Then (a) implies that $(\MM,\XX(0))$ is
$g_0$-distributed and (b) says that the requirement (\ref{tif}) is fulfilled. Moreover, 
defining the vorticity $w_t(B):=\int_{\rr\times \rr^2} m \indiq_B(x)g_t(dm,dx)$ for all $B\in\cB(\rr)$, 
we see from to \eqref{tif} and (\ref{eq:nablaLqbdFisher2}) that $(w_t)_{t\geq 0}$ satisfies (\ref{cfish2}),
which implies (\ref{cfas}) by Lemma \ref{efe}.
Finally, point (c) tells us that for all $\varphi \in C^2_b(\rr^2)$,
$$
\varphi(\XX(t))-\varphi(\XX(0))-\intot \int \tm K(\XX(s)-\tg_s) \cdot \nabla\varphi(\XX(s))g(d\tm,d\tg)ds 
- \nu \intot \Delta\varphi(\XX(s))ds
$$
is a martingale. This classically implies the existence of a $2D$-Brownian motion
$(\BB(t))_{t\geq 0}$ such that 
$$
\XX(t)=\XX(0)+\intot \int \tm K(\XX(s)-\tg_s)g(d\tm,d\tg) ds + \sigma \BB(t).
$$
>From the definition of $w_t$, we see that  $\int \tm K(\XX(s)-\tg_s)g(d\tm,d\tg)$
is nothing but $\intrd K(\XX(s)-x)w_s(dx)$. Hence $(\XX(t))_{t\geq 0}$ solves (\ref{NSDE}) as desired.

\vip

We thus only have to prove that $\QQ$ a.s. satisfies points (a), (b) and (c). For each 
$t\geq 0$,
we set $\QQ_t= \QQ \circ (m,\gamma(t))^{-1}$.

\vip

{\it Step 2.} 
We know from (\ref{chaosini}) that the sequence $G^N_0$ is $g_0$-chaotic, which implies that
$\QQ^N_0= \QQ^N\circ (m,\gamma(0))^{-1}$ goes weakly to $g_0$ in law (and thus in probability
since $g_0$ is deterministic), whence
$\QQ_0=g_0$ a.s. Hence $\QQ$ satisfies (a).

\vip

{\it Step 3.} Point (b) follows from Theorem \ref{th:levl3tH&tI} and Proposition \ref{cutoff}.
Indeed, recall that $G^N_t$ is the law of $(\MM_i^N,\XX^N_i(t))_{i=1,\dots,N}$.
Since the $\MM_i^N$ are i.i.d. and $r_0$-distributed by assumption (\ref{chaosini})
and since the system is exchangeable, it holds that $G^N_t \in \Ee_n(r_0)$ for all $t\geq 0$
and $r_0$ is supported in $[-A,A]$ still by (\ref{chaosini}). Next, Proposition \ref{cutoff}
implies that  $\sup_{N\geq 2} \tMk(G^N_t)<\infty$, which is equivalent, by exchangeability, to 
$\sup_{N\geq 2} \tMk(G^N_{t1})<\infty$. Finally, we know that
$N^{-1}\sum_{i=1}^N\delta_{(\MM_i^N,\XX^N_i(t))}$ goes weakly to $\QQ_t$ in law (by hypothesis), 
which classically implies (see e.g. Sznitman \cite{SSF}) that for all $j\geq 1$, $G^N_{tj}$ goes weakly to 
$\pi_{tj}$, where $\pi_t:=\LL(\QQ_t)$ and where $\pi_{tj}=\int_{\PPP(\rr\times\rr^2)}g^{\otimes j}\pi_t(dg)$.
We thus may apply Theorem \ref{th:levl3tH&tI} (for each $t\geq 0$) and deduce that
$\int_{\PPP(\rr\times\rr^2)} \tI(g)\pi_t(dg) \leq \liminf_N \tI(G^N_t)$. By the Fatou Lemma and by definition of
$\pi_t$, this yields
$$
\E\left[\int_0^T \tI(\QQ_s)ds\right] = \int_0^T \int_{\PPP(\rr\times\rr^2)} \tH(g)\pi_t(dg) dt
\leq  \int_0^T \liminf_N \tI(G^N_s)dt \leq \liminf_N \int_0^T \tI(G^N_s)dt.
$$
This last quantity is finite by Proposition \ref{cutoff}, so that
$\int_0^T \tI(\QQ_s)ds<\infty$ a.s.

\vip

{\it Step 4.} From now on, we consider some fixed $\cF: \PPP(\R\times C([0,\infty),\rr^2))\mapsto \rr$ 
as in point (c). We will check that $\cF(\QQ)=0$ a.s.
and this will end the proof.

\vip

{\it Step 4.1.} Here we prove that for all $N\geq 2$,
\begin{equation}\label{scc1}
\E \left[ (\cF(\QQ^N))^2  \right] \leq \frac {C_\cF} N.
\end{equation}
To this end, we recall that $\varphi \in C^2_b(\rr^2)$ is fixed 
and we apply the It\^o formula to (\ref{ps}): for all $i=1,\dots,N$, 
(here we use the convention that $K(0)=0$)
\begin{align*}
O^N_i(t):=&\varphi(\XX^N_i(t))
-\frac 1N \sum_j \MM_j^N \intot \nabla\varphi(\XX^N_i(s)) \cdot
K(\XX^N_i(s)-\XX^N_j(s)) ds - \frac {\sigma^2}2 \intot \Delta \varphi(\XX^N_i(s))ds\\
=& \varphi(\XX^N_i(0)) + \sigma \intot \nabla \varphi(\XX^N_i(s))d\BB^i_s.
\end{align*}
But one easily get convinced that 
\begin{align*}
\cF(\QQ^N)=& \frac 1 N \sum_{i=1}^N \psi(\MM_i^N)\varphi_1(\XX^N_i(t_1))\dots \varphi_k(\XX^N_i(t_k)) 
[O^N_i(t)-O^N_i(s)]\\
=&\frac \sigma N \sum_{i=1}^N \psi(\MM_i^N)\varphi_1(\XX^N_i(t_1))\dots \varphi_k(\XX^N_i(t_k)) 
\int_s^t \nabla \varphi(\XX^N_i(u)) d\BB^i_u.
\end{align*}
Then (\ref{scc1}) follows from some classical stochastic calculus, using that $0<t_1<\dots<t_k<s<t$, that
$\psi,\varphi_1,\dots,\varphi_k,\nabla\varphi$ are bounded and that
the Brownian motions $\BB^1,\dots,\BB^N$ are independent.

\vip

{\it Step 4.2.} Next we introduce, for $\e\in (0,1)$, the smoothed kernel $K_\e:\rr^2\mapsto\rr²$ defined
by $K_\e(x)=K(x \max(|x|,\e)/|x|)$. This kernel is continuous, bounded, verifies $K_\e(x)=K(x)$
as soon as $|x|\geq \e$ and $|K_\e(x)|\leq |K(x)|=|x|^{-1}$. We also introduce
$\cF_\e$ defined as $\cF$ with $K$ replaced by $K_\e$. Then one easily checks that $g\mapsto \cF_\e(g)$
is continuous and bounded from $\PPP(\rr\times C([0,\infty),\rr^2))$ to $\rr$. Since
$\QQ^N$ goes in law to $\QQ$, we deduce that for any $\e\in(0,1)$,
$$
\E[|\cF_\e(\QQ)|]=\lim_{N} \E[|\cF_\e(\QQ^N)|].
$$

{\it Step 4.3.} We now prove that for all $N\geq 2$, all $\e\in(0,1)$,
$$
\E[|\cF(\QQ^N)-\cF_\e(\QQ^N)|] \leq C_\cF \sqrt \e.
$$
Using that all the functions (including the derivatives) involved in $\cF$ are bounded and that
$|K_\e(x)-K(x)|\leq |x|^{-1}\indiq_{\{0<|x|<\e\}}$, we get
\begin{align} \label{tas}
|\cF(g)-\cF_\e(g)| \leq& C_\cF \int \int  |\tm| \int_0^t |\gamma(u)-\tg(u)|^{-1}
\indiq_{\{0<|\gamma(u)-\tg(u)|<\e\}} du g(d\tm,d\tg) g(dm,d\gamma) \\
\leq & C_\cF \sqrt \e \int \int  |\tm| \int_0^t |\gamma(u)-\tg(u)|^{-3/2}
\indiq_{\{\gamma(u)\ne\tg(u)\}}du g(d\tm,d\tg) g(dm,d\gamma). \nonumber
\end{align}
Thus
$$
|\cF(\QQ^N)-\cF_\e(\QQ^N)| \leq C_\cF 
\frac {\sqrt \e} {N^2} \sum_{i\ne j} |\MM_j^N|\int_0^t |\XX^N_i(u)- \XX^N_j(u)|^{-3/2} du,
$$
whence by exchangeability (and since  $|\MM_j^N| \leq A$ a.s. for all $j$ by (\ref{chaosini})),
$$
\E[|\cF(\QQ^N)-\cF_\e(\QQ^N)|] \leq C_\cF {\sqrt \e} \intot \E[|\XX^N_1(u)- \XX^N_2(u)|^{-3/2}] du.
$$
Denoting by $G^N_{u2}$ the two-marginal of $G^N_u$ and using Lemma \ref{lem:FishInteg} with $\gamma=3/2$
and $\beta=1$, we get
$$
\E[|\cF(\QQ^N)-\cF_\e(\QQ^N)|] \leq C_\cF {\sqrt \e} \intot \tI(G^N_{u2})du.
$$
  We conclude using Proposition \ref{cutoff} and that $\tI(G^N_{u2})\leq 2 \tI(G^N_u)$,
see Corollary~\ref{cor:fishdec}.  

\vip
{\it Step 4.4.} We next check that a.s.,
$$
\lim_{\e \to 0} |\cF(\QQ)-\cF_\e(\QQ)| = 0.
$$
Since $\QQ$ is the limit in law of $\QQ^N$ by assumption and since $\supp \QQ^N \subset [-A,A]\times 
C([0,T],\rr^2)$ a.s. thanks to (\ref{chaosini}), we deduce that $\supp \QQ \subset [-A,A]\times 
C([0,T],\rr^2)$ a.s. Hence $\supp \QQ_s \subset [-A,A]\times \rr^2$ a.s. for each $s\geq 0$.
Denote by $v_s(dx):=\int_\rr \QQ_s(dm,dx)$,  we have from Step 3 and Lemma \ref{lem:FishInteg2} that 
$\nabla v \in L^{2q/(3q-2)}([0,T],L^q(\R^2))$ for all $q\in [1,2)$ a.s., whence 
$$
\intot \intrd \intrd |x-y|^{-3/2} v_s(dx)v_s(dy) ds <\infty
$$
a.s. by Lemma \ref{efe}. Using now (\ref{tas}), we deduce that
\begin{align*}
|\cF(\QQ)-\cF_\e(\QQ)| \leq & C_\cF A \sqrt \e \,  
\int \int  \int_0^t |x(s)-\tx(s)|^{-3/2}ds \QQ (d\tm,d\tx) \QQ (dm,dx)\\
=& C_\cF A \sqrt  \e \,  \int_0^t \intrd \intrd  |x-y|^{-3/2} v_s(dx)v_s(dy) ds.
\end{align*}
The conclusion follows.

\vip

{\it Step 4.5.} We finally conclude: for any $\e \in (0,1)$, we write, using Steps 4.1, 4.2 and 4.3,
\begin{align*}
\E[|\cF(\QQ)|\land 1] \leq & \E[|\cF_\e(\QQ)|] + \E[|\cF(\QQ)-\cF_\e(\QQ)|\land 1] \\
= & \lim_N \E[|\cF_\e(\QQ^N)|] + \E[|\cF(\QQ)-\cF_\e(\QQ)|\land 1] \\
\leq & \limsup_N \E[|\cF(\QQ^N)|] + \limsup_N \E[|\cF(\QQ^N)-\cF_\e(\QQ^N)|] 
+ \E[|\cF(\QQ)-\cF_\e(\QQ)|\land 1] \\
\leq & C_\cF \sqrt \e +  \E[|\cF(\QQ)-\cF_\e(\QQ)|\land 1].
\end{align*}
We now make tend $\e\to 0$ and use that $\lim_\e \E[|\cF(\QQ)-\cF_\e(\QQ)|\land 1]=0$
thanks to Step 4.4 by dominated convergence. Consequently, 
$\E[|\cF(\QQ)|\land 1]=0$, whence $\cF(\QQ)=0$ a.s. as desired.
\end{proof}

\section{Well-posedness for the limit equation and its stochastic paths}\label{sec:blm}
\setcounter{equation}{0}

We first give the

\vip

\begin{preuve} {\it of Theorem \ref{th:wp}.} 
First, existence follows from Proposition \ref{consist}. 
Let $w_0$ satisfy (\ref{condic}), introduce $g_0$ satisfying (\ref{condic2}) such that (\ref{lienwg})
holds true as in Remark \ref{remtc}-(ii) and finally consider set $G_0^N:=g_0^{\otimes N}$, which satisfies
(\ref{chaosini}). Then Proposition \ref{consist}
implies the existence (in law) of a solution to the nonlinear SDE \ref{NSDE} associated to $g_0$ and 
such that (\ref{tif}) holds true. 
Defining $(w_t)_{t\geq 0}$ by (\ref{def:vorticityOFg}), Remark \ref{nsdeins2d} 
implies that $(w_t)_{t\geq 0}$ is a weak solution starting from $w_0$
to (\ref{NS2D}). Furthermore, we have seen in the proof of Proposition \ref{consist}, Step 1,
that $(w_t)_{t\geq 0}$ satisfies \eqref{cfish2}.

\vip

We now turn to uniqueness and renormalization, which we prove in several steps. We consider a weak 
solution $(w_t)_{t\geq 0}$ of \eqref{NS2D} satisfying \eqref{cfish2} and we put 
$\bar K (t,x) := (K \star w_t)(x)$.

\vip

{\sl Step 1. First Estimates. }
Because of \eqref{cfish2}, we know that a.e. in time, $w_s$ is a measurable function, and 
thanks to the $\MMM([0,T],\rr^2)$-weak continuity assumption,
we deduce that
\beqn\label{eq:bddL1w}
w \in L^\infty (0,T; L^1(\R^2)) \quad \forall \, T > 0. 
\eeqn
Also observe that \eqref{cfish2} and \eqref{eq:bddL1w} imply, thanks to Lemma \ref{efe}, that
\beqn
\label{eq:wLqtLpx}
 w \in L^{p/(p-1)}(0,T;L^p(\rr^2)) \quad \forall \; p\in(1,\infty), 
\forall \; T>0,
\eeqn
By definition of $K$ and by the Hardy-Littlewood-Sobolev inequality 
(of which a particular case is $||\intrd |.-y|^{-1}f(y)dy||_{L^{2p/(2-p)}}
\leq C_p ||f||_{L^p}$ for all $p\in(1,2)$, 
see e.g. \cite[Theorem 4.3]{LL}), we thus get
\beqn
\label{eq:barKLqtLpx}
\bar K  \in L^{p/(p-1)}(0,T;L^{2p/(2-p)}(\rr^2)) \quad \forall \; p\in(1,2), 
\forall \; T>0.
\eeqn
Similarly, \eqref{cfish2} and the Hardy-Littlewood-Sobolev inequality imply that
\beqn\label{bdd:dbbarK} 
\nabla_x \bar K = K * (\nabla_x w) \in L^{p/(p-1)}(0,T;L^p(\rr^2)) \quad \forall \; p\in(2,\infty), \quad 
\forall \; T>0.
\eeqn

{\sl Step 2. Continuity. } 
Consider a mollifier sequence $(\rho_n)$ on $\rr^2$ and introduce the
mollified function $w^n_t := w_t * \rho_n$. Clearly, $w^n  \in C([0,\infty), L^1(\rr^2))$.
Using \eqref{eq:wLqtLpx} and  \eqref{bdd:dbbarK}, a variant of the commutation
Lemma \cite[Lemma II.1 and  Remark 4]{dPL} tells us that 
\beqn\label{eq:weps}
\partial_t w^n - \bar K \cdot \nabla_x w^n   - \nu\Delta_x w^n = r^n, 
\eeqn
with
$$
r^n := (  \bar K \cdot \nabla_x w) * \rho_n -  \bar K \cdot \nabla_x
w^n \to 0 \quad\hbox{in}\quad L^1(0,T; L^1_{loc}(\rr^2)).
$$
The important point here is that $|\nabla_x \bar K| \,|\omega|  \in L^1((0,T)
\times \rr^2)$, thanks to~\eqref{bdd:dbbarK} and~\eqref{eq:wLqtLpx}. Remark
that the singularity of the Biot-Savard kernel is sharp for that property : it
will no longer be true if we increase the singularity. It is the first time
that this happens, all we have done before remains valid for a singularity like
$|x|^{-\gamma}$ with $\gamma \in (0,2)$.

As a consequence,  the chain rule applied to the smooth $w^n$ reads 
\beqn\label{eq:betaweps}
\partial_t \beta(w^n) =  \bar K \cdot \nabla_x \beta(w^n) + 
 \nu\Delta_x \beta(w^n) -  \nu\beta''(w^n) \, |\nabla_x w^n|^2 +
\beta'(w^n) \, r^n,
\eeqn
for any $\beta \in C^1(\rr) \cap W^{2,\infty}_{loc}(\R)$ such that $\beta''$ is piecewise continuous and  
vanishes
outside of a compact set. Because the equation \eqref{eq:weps} with $\bar K$
fixed is linear, the difference 
$w^{n,k} := w^{n} - w^{k}$ satisfies \eqref{eq:weps} with $r^n$ replaced by $r^{n,k} := r^{n} - r^{k} \to 0$
in $L^1(0,T;L^1_{loc}(\R^2)$ and then also \eqref{eq:betaweps} (with again $w^n$ and $r^n$ changed in 
$w^{n,k}$ and $r^{n,k}$). In that last 
equation, we  choose $\beta(s) = \beta_1(s)$ where $\beta_M(s) = s^2/2$ for $|s| \le M$,  
$\beta_M(s) = M \, |s| - M^2/2$ for $|s| \ge M$ and we obtain 
for any non-negative $\chi \in C^2_c(\R^d)$,
\bean
\int_{\R^2} \beta_1(w^{n,k}(t,x)) \, \chi(x) \, dx  
&\le&  \int_{\R^2} \beta_1(w^{n,k}(0,x))\, \chi(x)  \, dx +  \int_{0}^{t} \!  \int_{\R^2}  |r^{n,k}(s,x)| \, 
\chi(x) \, dxds  
\\
&&+ \int_{0}^{t} \!  \int_{\R^2}  \beta_1 (w^{n,k}(s,x)) \, \Bigl(  \nu \Delta
\chi(x) - 
\bar K(s,x) \cdot \nabla \chi (x)\Bigr) \, dxds  
\eean
where we have used that $\hbox{div}_x \, \bar K = 0$, that $|\beta_1'|
\le 1$  and that $\beta_1'' \ge 0$. 
Because 
$w_0 \in L^1$, we have $w^{n,k}(0) \to 0$
in $L^1(\R^2)$, and we deduce from the previous inequality, the convergence $r^{n,k}  \to 0$ in  
$L^1(0,T; L^1_{loc}(\rr^2))$,
the convergence  $ \beta_1(w^{n,k}) \bar K  \to 0$ in $L^1(0,T;L^1_{loc}(\R^2))$
(because $\beta_1(s)\leq |s|$, because $w^{n,k}\to 0$ in $L^{3}(0,T,L^{3/2}(\rr^2))$ 
by \eqref{eq:wLqtLpx} with $p=3/2$ and since $\bar K \in L^{6}(0,T;L^3(\rr^2)) 
\subset L^{3/2}(0,T;L^{3}(\rr^2))$ by \eqref{eq:barKLqtLpx} with $p=6/5$), that
$$
\sup_{t \in [0,T]} \int_{\R^2} \beta_1(w^{n,k}(t,x)) \, \chi(x) \, dx  
\, \mathop{\longrightarrow}_{n,k \to \infty} \, 0.
$$
Since $\chi$ is arbitrary, we deduce that there exists $\bar w \in C([0,\infty); L^1_{loc}(\R^2))$ so that 
$w^n \to \bar w$ in $C([0,\infty); L^1_{loc}(\R^2))$, with the topology of
uniform convergence on any compact subset in time. Together with the
convergence $w^n \to w$ in 
$C([0,\infty);\MMM(\R^2))$ we deduce  
that $w=\bar w$ and with the same convention for the notion of convergence on
$[0,\infty)$
\beqn\label{eq:wcont}
w^n \to  w \quad \hbox{in} \quad C([0,\infty); L^1(\R^2)).
\eeqn

{\sl Step 3. Additional estimates. } 
We come back to \eqref{eq:betaweps}, which implies, for all $0<t_0<t_1$, all $\chi\in C^2_c(\rr^2)$,
\bear\label{eq:betawnbis}
\int_{\R^2} \beta(w^{n}_{t_1}) \, \chi  \, dx  +  \nu  \int_{t_0}^{t_1} \! 
\int_{\R^2} 
\beta''(w^n_s) \, |\nabla_x w^n_s|^2 \, \chi\, dxds
=  \int_{\R^2} \beta(w^{n}_{t_0}) \, \chi  \, dx \quad
\\ \nonumber
+ \int_{t_0}^{t_1} \!  \int_{\R^2} \Bigl\{\beta'(w^n _s) \, r^n \, \chi +  \beta (w^n_s) \,   
\nu \Delta \chi  - 
\beta (w^n_s) \,  \bar K  \cdot \nabla \chi \Bigr\} \, dxds.  
\eear
Choosing $0 \le \chi \in C^2_c(\R^2)$ and 
$\beta \in C^1(\rr) \cap W^{2,\infty}_{loc}(\R)$ such that $\beta''$ is non-negative and  vanishes
outside of a compact set,  and passing to the limit as $n\to\infty$ 
(see Step 2 for the details of a similar convergence), we get 
\bean
&&\int_{\R^2} \beta(w_{t_1}) \, \chi  \, dx  \le
  \int_{\R^2} \beta(w_{t_0}) \, \chi  \, dx
  +  \int_{t_0}^{t_1} \!  \int_{\R^2} \beta (w_s) \,  \Bigl\{   \nu \Delta \chi 
- 
 \bar K  \cdot \nabla  \chi \Bigr\} \, dxds.  
\eean
It is not hard to deduce, by approximating $\chi\equiv 1$ by a well-chosen sequence $\chi_R$, using that
$\int_{t_0}^{t_1} \intrd |w_s(x)| \indiq_{\{|x|\geq R\}} dx$ clearly tends
to $0$ as $R \to \infty$ and that $\beta$ is sublinear, that
\bean
\int_{\R^2} \beta(w_{t})   \, dx  \le
  \int_{\R^2} \beta(w_{t_0}) \, dx \quad \forall \, t \ge t_0 \ge 0. 
  \eean
Finally,  letting $\beta(s) \to |s|^p/p$ and then
$p \to \infty$,  we get
\beqn\label{eq:wLp}
\| w (t,.) \|_{L^p} \le \| w (t_0,.) \|_{L^p}, \quad \forall \, p \in [1,\infty], \,\, 
\forall \, t \ge t_0 \ge 0. 
\eeqn
Taking now $\beta = \beta_M$ in \eqref{eq:betawnbis}, we have
\bean
\int_{\R^2} \beta_M(w^{n}_{t_1}) \, \chi  \, dx  +  \nu  \int_{t_0}^{t_1} \! 
\int_{\R^2} 
\indiq_{\{|w^n_s|\leq M\}}|\nabla_x w^n_s|^2 \, \chi \, dxds
=  \int_{\R^2} \beta_M(w^{n}_{t_0}) \, \chi \, dx  
\quad\\ \nonumber
\qquad+ \int_{t_0}^{t_1} \!  \int_{\R^2} \Bigl\{ \beta_M'(w^n _s) \, r^n \, \, \chi +  \beta_M (w^n_s) \,   
\nu \Delta \, \chi  - 
\beta_M (w^n_s) \,  \bar K  \cdot \nabla \, \chi \Bigr\} \, dxds, 
\eean
Similarly as above we first make tend $n\to\infty$, then we 
approximate $\chi\equiv 1$ by a well-chosen sequence $\chi_R$ and make tend $R\to \infty$, and finally we
take the limit as $M\to \infty$: this yields
\beqn\label{eq:wL2}
\int_{\R^2} w^2_{t_1}  \, dx  +  \nu  \int_{t_0}^{t_1} \!  \int_{\R^2} |\nabla_x
w_s|^2   \, dxds
\le  \int_{\R^2} w^2_{t_0}    \, dx  \quad \forall \, t _1 \ge t_0 \ge 0. 
\eeqn
Using \eqref{eq:wLqtLpx}, \eqref{eq:wLp} and \eqref{eq:wL2}, we deduce that for all $0<t_0<T$,
\beqn\label{eq:uniqvortexregL2nabla}
\forall \, p \in [1,\infty), \;\; w \in L^\infty(t_0,T; L^p(\rr^2))
\quad \hbox{and}\quad 
\nabla_x w \in L^2((t_0,T) \times \rr^2).
\eeqn
It is easily checked, using the H\"older inequality, that $||\bar K_t||_{L^\infty}\leq
C(||w_t||_{L^1}+||w_t||_{L^3})$. Hence, $\bar K \in L^\infty(t_0,T; L^\infty(\rr^2))$.
We thus have
\beqn\label{eq:heatw}
\partial_t w + \Delta_x w = \bar K \cdot \nabla_x w \in  L^2((t_0,T) \times \rr^2),
\quad \forall \, t_0 >0
\eeqn
so that the maximal regularity of the heat equation in $L^2$-spaces (see Theorem X.11 stated in 
\cite{BrezisBook} and the quoted reference) provides the bound 
\beqn\label{eq:wMaximal}
w \in L^2(t_0,T; H^2(\R^2)) \cap L^\infty(t_0,T; H^1(\R^2)), \quad  \forall
\,t_0 >0.
\eeqn
We emphasize that starting form the bound $\bar K \cdot \nabla w \in L^2(
L^2((t_0,T) \times \rr^2)$ and when $w_{t_0} \in H^1$, the maximal regularity implies the above bound
on the time interval $[t_0,\infty)$. But thanks to~\eqref{eq:uniqvortexregL2nabla}, we can find $t_0$
arbitrarily close to $0$ such that $w_{t_0/2} \in H^1$, and this implies that
\eqref{eq:wMaximal} is correct for any $t_0>0$. 

Thanks to \eqref{eq:wMaximal}, an interpolation 
inequality and the Sobolev inequality, we deduce that
$\nabla_x w \in L^p ((t_0,T) \times \R^2)$ for any $1 < p < \infty$, whence $ \bar K \cdot \nabla_x
w \in L^p ((t_0,T) \times \R^2)$, for all $t_0 >0$. Then the maximal regularity
of the heat equation in $L^p$-spaces  
(see Theorem X.12 stated in \cite{BrezisBook} and the quoted references)  provides the bound 
\beqn\label{eq:wMaximalp}
\partial_t w, \nabla_x w  \in L^p((t_0,T) \times \R^2) , \quad \forall \, t_0 >0
\eeqn
and then the Morrey inequality implies $w \in C^{0,\alpha}((t_0,T) \times \R^2)$
for any $0< \alpha < 1$, and any $t_0 >0$. 
All together 
we conclude with  
$$
w  \in C([0,T) ; L^1(\rr^2)) \cap C((0,T) ; L^\infty(\rr^2)),
$$
which is nothing but \eqref{bdd:dbK3}.

\vip

{\sl Step 4. Uniqueness.} At this stage, we thus have shown that any weak solution to (\ref{NS2D}) 
satisfying (\ref{cfish2}) meets the assumptions
of \cite{Brezis} (which improves, thanks to very quick but smart arguments, the uniqueness 
result stated in \cite[Theorem B]{BenArtzi}). Such a solution is thus unique.

\vip

{\sl Step 5. Renormalization. } We end the proof by showing \eqref{eq:betaw}. 
Let thus $\beta \in C^1(\rr) \cap W^{2,\infty}_{loc}(\R)$ such that $\beta''$ is piecewise continuous and  
vanishes outside of a compact set.
Thanks to 
\eqref{eq:uniqvortexregL2nabla}, we can pass to the 
limit in the similar identity as \eqref{eq:betawnbis} obtained for time dependent test functions
$\chi \in C^2_c([0,\infty) \times \R^2)$ and we get 
\bear\label{eq:betawnbis2}
&& \nu  \int_{t_0}^\infty \!  \int_{\R^2} \beta''(w_s) \, |\nabla_x w_s|^2 \,
\chi\, dxds
=   \int_{\R^2} \beta(w_{t_0}) \, \chi  \, dx 
\\ \nonumber
&&\qquad+ \int_{t_0}^\infty \!  \int_{\R^2} \beta (w_s) \,   \Bigl\{  \nu \Delta
\chi  - 
 \bar K  \cdot \nabla \chi -  \partial_t \chi\Bigr\} \, dxds.  
\eear
When moreover $\chi \ge 0$ and $\beta'' \ge 0$, we can pass to the limit $t_0 \to 0$ thanks to 
monotonous convergence in the first term, 
the continuity property \eqref{eq:wcont} in the second term
and the Lebesgue dominated convergence 
theorem in the third term (recall that $\beta$ is sublinear and that $|w|(1+|\bar K|)$
belongs to $L^1(0,T;L^1(\rr²))$ because $w \in L^{3}(0,T;L^{3/2}(\R^2))$ by 
(\ref{eq:wLqtLpx}) with
$p=3/2$ and $\bar K \in L^{6}(0,T;L^{3}(\R^2)) \subset L^{3/2}(0,T;L^{3}(\R^2))$ by
(\ref{eq:barKLqtLpx}) with $p=6/5$) and we get
\bear\label{eq:betawnbis3}
&& \nu  \int_{0}^\infty \!  \int_{\R^2} \beta''(w_s) \, |\nabla_x w_s|^2 \,
\chi\, dxds
=   \int_{\R^2} \beta(w_0) \, \chi  \, dx 
\\ \nonumber
&&\qquad+ \int_{0}^\infty \!  \int_{\R^2} \beta (w_s) \,   \Bigl\{  \nu \Delta
\chi  - 
 \bar K  \cdot \nabla \chi -  \partial_t \chi\Bigr\} \, dxds.  
\eear
With the new bound on the first term provided by \eqref{eq:betawnbis3}, we can pass to the limit 
as $t_0 \to 0$ in \eqref{eq:betawnbis2}
and get \eqref{eq:betawnbis3}
for arbitrary test functions $\chi$ and renormalizing functions $\beta$ (i.e. without the assumptions
that $\chi$ and $\beta''$ are non-negative). This is nothing but 
\eqref{eq:betaw} in the distributional sense.
\end{preuve}

We now turn to the well-posedness of the nonlinear SDE \eqref{NSDE}.

\begin{preuve} {\it of Theorem \ref{wpnsde}.} Let thus $g_0$ satisfy \eqref{condic2}.
Here again, Proposition \ref{consist} (e.g. with the choice $G_0^N=g_0^{\otimes N}$) 
shows the existence (in law) of a solution to the nonlinear SDE \eqref{NSDE}
such that (\ref{tif}) holds true. Defining $(w_t)_{t\geq 0}$ by (\ref{def:vorticityOFg}), 
Remark \ref{nsdeins2d}  implies that $(w_t)_{t\geq 0}$ is a weak solution 
to (\ref{NS2D}). Furthermore, we have seen in the proof of Proposition \ref{consist}, Step 1,
that $(w_t)_{t\geq 0}$ satisfies \eqref{cfish2}. Hence $(w_t)_{t\geq 0}$ is uniquely determined by
Theorem~\ref{th:wp}. We will check below the pathwise uniqueness for the linear equation
\begin{align} \label{linpro}
\XX(t) = \XX(0)  + \int_0^t \bar K_s (\XX(s)) \,ds + \sigma \BB_t,
\end{align}
where $\bar K_s = K \ast w_s$. This will end the proof. Indeed, pathwise uniqueness for (\ref{NSDE}) will 
immediately follow (consider two solutions $\XX,\YY$ to \eqref{NSDE} associated to the same Brownian
motion $\BB$ and the same $(\MM,\XX(0))$, observe that both satisfy (\ref{linpro}) 
with the same Brownian motion, so that they coincide).
Now existence in law and pathwise uniqueness classically imply strong existence by the
Yamada-Watanabe theorem \cite{yw}. 

\vip

For the weak uniqueness to (\ref{linpro}), we might refer to \cite{Figalli} which assume
that $\bar K \in L^2_{t,x,loc}$. For the pathwise uniqueness to (\ref{linpro}), we might use 
\cite{LeBrisLions}, who assume that $\bar K \in L^1([0,T],W^{1,1}(\rr^2))$. But we shall give 
here an alternative proof for pathwise uniqueness,
which is well-suited to our initial (entropic) distribution.
We adapt to our context the method of \cite{CrippaDeLellis} concerning deterministic ODEs with low 
regularity vector-field.

\vip

We thus assume that we have two solutions $\XX$ and $\YY$ to \eqref{linpro} with the same Brownian motion
$\BB$, the same value of $(\MM,\XX(0))$ and the same vector-field $\bar K$. Then, obviously,
$$
\XX(t) - \YY(t)= \intot (\bar K_s (\XX(s))-\bar K_s(\YY(s))) \,ds,
$$
so that for any $\delta>0$,
$$
\log(\delta+ |\XX(t) - \YY(t)|)\leq \log \delta + \intot \frac{|\bar K_s (\XX(s))-\bar K_s(\YY(s))|}
{\delta + |\XX(s) - \YY(s)|} \,ds
$$
and thus
$$
 \E \left[ \log (\delta  + \sup_{0\le s\le t} |\XX(s) - \YY(s)|)\right] \leq  
\log \delta + \intot \E \left[
\frac{|\bar K_s(\XX(s)) - \bar K_s(\YY(s))|}{\delta  + |\XX(s) - \YY(s)|}  \right] 
\,ds.
$$

We will use the following facts: for a measurable function $f$ on $\rr^2$, define the Hardy-Littlewood maximal function
$M f(x)= \sup_{r >0} |B_r(x)|^{-1} \int_{B_r(x)} |f(y)|dy$, where
$B_r(x)$ is the ball centered at $x$ with radius $r$. Then, for a.e. $x,y \in \rr^2$, see \cite[Corollary 4.3 with $\alpha=0$]{Aalto}
\begin{eqnarray} \label{ineq1}
|f(x) - f(y) | \le C \left[ M\nabla f(x) + M\nabla f(y)\right] |x - y|
\end{eqnarray}
and for all $p\in [1,\infty]$, see \cite[Theorem 1 in Chapter 1]{Stein},
\begin{eqnarray}\label{ineq2}
\| M f  \|_{L^p} \le C_{p} \| f \|_{p}, \qquad \forall p\in [1,\infty].
\end{eqnarray}
Using (\ref{ineq1}), we obtain
$$
 \E \left[ \log (\delta  + \sup_{0\le s\le t} |\XX(s) - \YY(s)|)\right] \leq 
\log \delta + 
C \intot \E \left[
|M \nabla_x \bar K_s(\XX(s)) + M \nabla_x  \bar K_s(\YY(s))|  \right]  \,ds.
$$
Denoting now by $v^1$ (resp. $v^2$) the law of $\XX(t)$ (resp. $\YY(t)$), we remember that (\ref{tif})
(which is assumed for both solutions)
and Lemma \ref{lem:FishInteg2} imply that for $i=1,2$
\begin{equation}\label{ineq3}
\forall p \in [1,+\infty), \;\; v^i_t  \in L^{p/(p-1)}([0,T], L^p(\rr^2)).
\end{equation}
Using (\ref{ineq3}) with $p=3/2$, (\ref{ineq2}) 
and the estimate \eqref{bdd:dbbarK} with $p=3$,
\begin{eqnarray*}
\intot \E \left[ M \nabla_x \bar K_s(\XX(s)) \right] \,ds &=& \intot \intrd M\nabla_x \bar K_s (x) v^1_s(x)dx\\
&\le& \int_0^t \| M \nabla_x \bar K_s \|_{L^3} \| v^1_s\|_{L^{3/2}} \,ds  \\
&\leq & C \int_0^t \| \nabla_x \bar K_s \|_{L^3} \| v^1_s\|_{L^{3/2}} \,ds  \\
& \le &  C \|\nabla_x \bar K  \|_{L^{3/2}([0,t],L^{3}(\R^2))} \|v^1 \|_{L^{3}([0,t],L^{3/2}(\R^2))} < \infty.
\end{eqnarray*}
Handling the same computation for $\YY$, we get that
$$
\E \left[ \log (\delta  + \sup_{0\le s\le t} |\XX(s) - \YY(s)|)\right] \leq \log \delta
+ C_t,
$$
where the constant $C_t$ is independent of $\delta$. From that and the fact that 
$u \mapsto \log u$ is increasing, setting $\ZZ_t := \sup_{0 \le s \le t} |\XX(s) - \YY(s)|$, we can
estimate for any $\eps >0$
\bean
\Prob (\ZZ_t > \eps) \, \log (1 + \eps \, \delta^{-1}) + \log \delta
&=& \Prob (\ZZ_t \le \eps) \, \log \delta + \Prob (\ZZ_t > \eps) \, \log (\delta + \eps)
\\
&=& \E \Bigl( {\bf 1}_{ \{\ZZ_t \le \eps\} } \, \log \delta +{\bf 1}_{ \{\ZZ_t >\eps \}}  
\, \log (\delta + \eps) \Bigr) 
\\
&\le& \E \Bigl(  \log (\delta + \ZZ_t )  \Bigr) \\
&\le&  \log \delta + C_t.
\eean
We have proved  
$$
\Prob(\sup_{0\le s\le t} |\XX(s) - \YY(s)| \ge \eps) \le \frac{C_t}{\log(1 + \eps
\delta^{-1})}.
$$
Letting $\delta \to 0$, we obtain $ \Prob(\sup_{0\le s\le t} |\XX(s) - \YY(s)|
\ge \eps)=0$.
Pathwise uniqueness is proved.

\vip

It remains to prove \eqref{eq:EntropyEq}. 
We denote by $g_t\in\PPP(\rr\times\rr^2)$  the law of $(\MM,\XX(t))$
and by $w_t\in \MMM(\rr^2)$ the associated vorticity, see (\ref{def:vorticityOFg}). 
Since $(\MM,(\XX(t))_{t\geq 0})$ has been obtained by passing to the limit in the particle system
\eqref{ps}, we deduce from Theorem \ref{th:mr}, Theorem \ref{th:levl3tH&tI} and Lemma \ref{cutoff}
that 
\beqn\label{eq:bdfm}
\sup_{[0,T]}\tH(g_t) < \infty, \quad \sup_{t \in [0,T]} \tMk(g_t) < \infty, \quad \int_0^T \tI(g_s) \, ds < \infty.
\eeqn
We call $r_0\in\PPP(\rr)$ the law of $\MM$ and for $m\in\rr$, we denote by $f^m_t$ 
the law of $\XX(t)$ knowing that $\MM=m$. We then have $g_t(dm,dx)=r_0(dm)f^m_t(dx)$ for all $t\geq 0$.
Thanks to It\^o calculus, $(f^m_t)_{t\geq 0}$ clearly belongs to $C([0,T]; \PPP(\R^2))$ 
(because $t\mapsto \XX(t)$ is a.s. continuous) and is a weak solution, for $m\in \rr$ fixed, to 
\beqn\label{eq:PDEfm}
\partial_t f^m = \nu \Delta_x f^m + \bar K \cdot \nabla_x f^m
\eeqn
where $\bar K_t = K \star w_t$. Using the definitions of $\tH,\tMk,\tI$, we deduce from 
\eqref{eq:bdfm} that for $r_0$-almost every $m\in\R$, for all $t\geq 0$,
\beqn\label{eq:bdfm2}
H(f^m_t) < \infty, \quad M_k(f^m_t) < \infty, \quad \int_0^t I(f^m_s) \, ds < \infty.
\eeqn
The Fisher information bound in \eqref{eq:bdfm2}  implies, by Lemma \ref{lem:FishInteg1}, that 
$$
 \nabla_x f^m   \in L^{2q/(3q-2)}(0,T, L^q(\rr^2))\quad  \forall \; q\in [1,2),\quad
\forall \; T>0.
$$ 
Then we use the same arguments as in the proof of (\ref{eq:betaw}) in 
Theorem \ref{th:wp} (which was entirely based on such an estimate plus an estimate saying that 
$(f^m_t)_{t\geq 0}$ belongs to $C([0,T]; \PPP(\R^2))$):
for any $t > 0$, any $\beta \in C^1(\rr) \cap W^{2,\infty}_{loc}(\R)$ such that $\beta''$ is 
piecewise continuous and  vanishes
outside of a compact set and any $\chi \in C^2_c(\rr^2)$,
\begin{eqnarray*}
&&\int_{\R^2} \beta(f^m_t) \chi \, dx 
+ \nu  \int_{0}^t \!  \int_{\R^2} \beta''(f^m_s) \, |\nabla f^m_s|^2  \chi\, 
dxds\\
&=&   \int_{\R^2} \beta(f^m_0) \chi\, dx    + \int_0^t\intrd \beta(f^m_s)
[\nu \Delta \chi - \bar K_s . \nabla \chi  ]dxds.
\end{eqnarray*}
Assume now additionally that $\beta''\geq 0$ and that $\beta(0)=0$.
Considering an
increasing sequence of uniformly bounded 
non-negative functions $\chi_k\in C^2_c(\rr^2)$ so that $\chi_k(x)=1$ for $|x|\leq k$,
it is not hard to deduce that (use the monotonous convergence theorem for the second term, 
the dominated convergence theorem and
that $|\beta(f^m_t)|+|\beta(f^m_0)| \leq C (f^m_t+f^m_0) \in L^1(\rr^2)$ 
for the first and third terms and finally, for the last term, the dominated convergence theorem and the fact
that $|\beta(f^m)|(1+|\bar K|) \in L^1([0,T]\times \rr^2)$ because $\bar K \in L^6(0,T;L^{3}(\rr^2))
\subset L^{3/2}(0,T;L^3(\rr^2))$ by \eqref{eq:barKLqtLpx} with $p=6/5$ and because  
$f^m \in L^3(0,T;L^{3/2}(\rr^2))$ thanks to the Fisher information estimate in (\ref{eq:bdfm2}) and 
Lemma \ref{lem:FishInteg2} with $p=3/2$),
\begin{eqnarray*}
\int_{\R^2} \beta(f^m_t) \, dx 
+ \nu  \int_{0}^t \!  \int_{\R^2} \beta''(f^m_s) \, |\nabla f^m_s|^2  dxds
=   \int_{\R^2} \beta(f^m_0) \, dx.
\end{eqnarray*}
We apply this with $\beta_p: \R_+ \to \R$ defined by
$\beta_p''(s) := (1/s) \, {\bf 1}_{\{s\in[1/p,p]\}}$, $\beta_p(0) = \beta_p(1) = 0$ and let $p\to \infty$.
The second term tends to $\nu \int_0^t I(f^m_s)ds$ 
as $p\to\infty$ by monotonous convergence. The first and third terms tend to $H(f^m_t)$ and $H(f^m_0)$
by monotonous convergence, 
because $0\leq -\beta_p(s)\indiq_{s\in [0,1]}$ increases to $-s \log s \indiq_{s\in [0,1]}$
while $0\leq \beta_p(s)\indiq_{s\in [1,\infty)}$ increases to $s \log s \indiq_{s\in [1,\infty)}$.
In fact, it can be checked that $\beta_p(s) = s \ln s +(1-s)/p$ if $s \in [ 1/p,p]$.
We finally get
$$
H(f^m_t) + \nu  \int_{0}^t  I(f^m_s) \, ds = H(f^m_0) .
$$
Integrating this equality against $r_0(dm)$ leads us to (\ref{eq:EntropyEq}).
\end{preuve}

\section{Conclusion}\label{sec:concl}
\setcounter{equation}{0}

It only remains to put together all the intermediate results.

\begin{preuve} {\it of Theorem \ref{th:mr}.}
Let us consider, for each $N\geq 2$, a family $(\MM_i^N,\XX_i^N(0))_{i=1,\dots,N}$
of $\rr\times\rr^2$-valued random variables. Assume that \eqref{chaosini} holds true for some $g_0$.
For each $N\geq 2$, consider the unique solution $(\XX^N_i(t))_{i=1,\dots,N,t\geq 0}$ to (\ref{ps}) and define
$\QQ^N := \frac1N \sum_{i=1}^N \delta_{(\MM_i^N,(\XX^N_i(t))_{t\geq 0})}$.
As shown in Lemma \ref{tight}, the family $\{\LLL(\QQ^N),N\geq 2\}$ is tight in 
$\PPP(\PPP(\rr\times C([0,\infty),\rr^2)))$. 
Proposition \ref{consist} shows that any (random) limit point $\QQ$
of this sequence belongs a.s. to $\cS$, the set of all probability measures 
$g\in\PPP(\rr\times C([0,\infty),\rr^2))$ such that $g$ is the law of
$(\MM,(\XX(t))_{t\geq 0})$ with $(\XX(t))_{t\geq 0}$ solution to the nonlinear SDE (\ref{NSDE}) 
satisfying that, denoting by $g_t\in
\PPP(\rr\times \rr^2)$ the law of $(\MM,\XX(t))$, (\ref{tif}) holds true. 
But Theorem \ref{wpnsde} implies that $\cS$ is reduced to one point $\cS=\{g\}$.
All this implies that $\QQ^N$ tends in law to $g$ as $N\to \infty$:
the sequence $(\MM_i^N,(\XX^N_i(t))_{t\geq 0})$ is $(\MM,(\XX(t))_{t\geq 0})$ -chaotic.

\vip
The last point follows thanks to the fact that all the circulations are
bounded by $A$:
we know that $\QQ^N$ goes in probability to $g$, in $\PPP(\R\times C([0,\infty),\R^2))$.
We also know that $\WW^N = \Phi(\QQ^N)$ and $w= \Phi(g)$, where 
$\Phi: \PPP(\R\times C([0,\infty),\R^2)) \mapsto C([0,\infty),\MMM(\R^2))$ is defined by
$(\phi(q))_t(B)= \int_{\R\times C([0,\infty),\R^2)} m \indiq_{\{\gamma(t)\in B\}} q(dm,d\gamma)$ for all $B\in\cB(\R^2)$.
A slightly tedious but straightforward study shows that this map is continuous on the subset
of all $q\in \PPP(\R\times C([0,\infty),\R^2))$ such that supp $q \subset [-A,A]\times C([0,\infty),\R^2)$.
The conclusion follows, since both $\QQ^N$ and $g$ a.s. belong to this subset by (\ref{chaosini}).
\end{preuve}

Finally, we give the proof of Theorem~\ref{th:mr2} on entropy chaos and strong convergence by adapting a 
trick introduced in \cite{MMKacProg} for the Boltzmann equation.

\begin{preuve}  {\it of Theorem \ref{th:mr2}.}
Recall that $G^N_t$ stands for the law of $(\MM_i^N,\XX_i^N(t))_{i=1,\dots,N}$, that $g_t$ is the law
of $(\MM,\XX(t))$, that we assume (\ref{chaosini}) and additionally that
$\lim_N \tH(G^N_0)=\tH(g_0)$. 

\vip

{\it Point (i). }
It readily follows from Theorem \ref{th:mr} that for each $t\geq 0$, 
$G^N_t$ is $g_t$-chaotic (in the sense of Kac) so that in particular, $(\MM_1^N,\XX_1^N(t))$ 
goes in law to $g_t$. It remains to prove that $\lim_N \tH(G^N_t)=\tH(g_t)$.
We first recall that from \eqref{entropGN} and the 
remark at the beginning of the proof of Proposition~\ref{cutoff}
$$
\forall \, t \ge 0, \qquad \tH(G^N_t) + \nu\int_0^t \tI(G^N_s) \, ds \le  
\tH(G^N_0),
$$
whence
$$
\limsup_N \{\tH(G^N_t) + \nu\int_0^t \tI(G^N_s) \, ds \} \le \limsup_N
\tH(G^N_0) = \tH(g_0).
$$
On the other hand,
applying Theorem~\ref{th:levl3tH&tI} (see Step 3 of the proof of Proposition \ref{consist} 
for similar considerations), we get 
$$
\liminf_N \tH(G^N_t) \ge \tH(g_t),
\quad
\liminf_N  \int_0^t \tI(G^N_s) \, ds \ge \int_0^t \tI(g_s) \, ds.
$$
Using that $\tH(g_t) +  \nu \int_0^t \tI(g_s) \, ds=\tH(g_0)$
by \eqref{eq:EntropyEq}, we easily conclude that for all $t\geq 0$,
$$
\lim_N \tH(G^N_t)= \tH(g_t) , \quad
\lim_N \int_0^t \tI(G^N_s) \, ds = \int_0^t \tI(g_s) \, ds
$$
as desired.

\vip
 
{\it Point (ii).} Denote by $r_0$ the law of $\MM$, recall that $r_0$ is supported in $\AA=[-A,A]$ and 
that the $\MM_i^N$ are i.i.d. and $r_0$-distributed.
For $j= 1,\dots,N$, we denote by $G^N_{tj}$ the $j$-th marginal of $G^N_t$ (that is, the
law of $(\MM_i^N,\XX_i^N(t))_{i=1,\dots, j}$), and by $F^{N,M}_{tj}$ the law of $(\XX_i^N(t))_{i=1,\dots, j}$
knowing that $(\MM_i^N)_{i=1,\dots,j}=M$ for any given $M\in\AA^j$. Then we have the disintegration formula
$G^N_{tj}(dM,dX)=r_0^{\otimes j}(dM)F^{N,M}_{tj}(dX)$. We also disintegrate $g_t(dm,dx)=r_0(dm)f^m_t(dx)$.

\vip

Using first Corollary \ref{cor:fishdec} (since $G^N_{tj}\to g_t^{\otimes j}$ weakly as $N\to\infty$ 
because $G^N_t$ is $g_t$-chaotic and since $\sup_N \tMk(G^N_t)<\infty$) and then that
$\lim_N \tH(G^N_t)= \tH(g_t)$ by Step 1, we have, for any $j\geq 1$, 
$$
\tH(g_t^{\otimes j}) \le \liminf_N  \tH(G^N_{t}) 
\limsup_N  \tH(G^N_{t})  = \tH(g_t)=\tH(g_t^{\otimes j}),
$$
so that,  for any $j \ge 1$, $\tH(G^N_{tj})  \to \tH(g_t^{\otimes j})$. 

Introducing artificially 
$$Q^N_{tj}(dM,dX)=r_0^{\otimes j}(dM) \left(\frac12  F^{N,M}_{tj}(dX) + \frac 12 
\prod_{i=1}^j f^{m_i}_{t}(dx_i)\right)= \frac 12 G^{N}_{tj}(dM,dX)+ \frac 12 g_t^{\otimes j}(dM,dX)
$$ 
(here we use the notation $X=(x_1,\dots,x_j)$ and $M=(m_1,\dots,m_j)$), it obviously holds that 
$Q^N_{tj}$ goes weakly to $g_t^{\otimes j}$ so that by lower semi-continuity, 
$\liminf_N \tH(Q^N_{tj}) \geq \tH(g_t^{\otimes j})$. We deduce that
$\limsup_N \left[\frac12 \tH(G^{N}_{tj}) + \frac12 \tH(g_t^{\otimes j}) - \tH(Q^N_{tj})  \right]  \leq 0$,
whence, by convexity of $\tH$,
$$
\limsup_N \left[\frac12 \tH(G^{N}_{tj}) + \frac12 \tH(g_t^{\otimes j}) - \tH(Q^N_{tj})  \right]= 0.
$$
Using the disintegration formulae and the definition of $\tH$, this rewrites
$$
\lim_N \int_{\R^j} r_0^{\otimes j} (dM) \left\{  \frac 12 H \bigl( F^{N,M}_{tj} \bigr) +  
\frac 12 H \Bigl( \prod_{i=1}^j f^{m_i}_{t} \Bigr) -  H \Bigl(  \frac 12 F^{N,M}_{tj} + \frac 12 
\prod_{i=1}^N f^{m_i}_{t} \Bigr)  \right\}=0.
$$
By the strict convexity of $s \mapsto s \, \log s$, this classically implies,
see for instance \cite{BrezisConvex} (all this can be rewritten as the integral against 
$r_0^{\otimes j} (dM)dX$ of a non-negative function), that from any (not relabelled) 
subsequence we can extract a (not relabelled) such that
\beqn\label{eq:cvgceaeGNj}
F^{N,M}_{tj}(X)  \to \prod_{i=1}^j f^{m_i}_{t}(x_i) \quad\hbox{for} \quad 
r_0^{\otimes j}\hbox{-a.e.}\,  M \in \R^j, \ 
\hbox{Lebesgue-a.e.} \, X \in (\R^2)^j.
\eeqn
On the other hand, the estimate established in Proposition~\ref{cutoff} together with Lemma~\ref{ieth} and Corollary~\ref{cor:fishdec} 
imply that $C_k + \tMk(G^N_{tj})+ \tH(G^N_{tj}) \leq 2 \, (C_k + \tMk(G^N_{t})+ \tH(G^N_{t}) ) 
\leq C$, which rewrites
$$
\forall \, N \ge 1 \qquad \int_{(\AA \times \R^2)^j}(\langle X \rangle^k + \log F^{N,M}_{tj}(X)) \, 
F^{N,M}_{tj}(X) \, r_0^{\otimes j}(dM) dX \le C.
$$
The Dunford-Pettis theorem thus implies that
\beqn\label{eq:equivL1GNj}
F^{N,M}_{tj}(X)\quad\hbox{is weakly compact in}\quad L^1((\AA \times \R^2)^j; r_0^{\otimes j}(dM) dX).
\eeqn
It is then a well-known application of the Egorov theorem that 
\eqref{eq:cvgceaeGNj} and \eqref{eq:equivL1GNj} imply that
$$
F^{N,M}_{tj}(X) \to  \prod_{i=1}^j f^{m_i}_{t}(x_i) 
\quad\hbox{strongly  in}\quad L^1((\AA \times \R^2)^j; r_0^{\otimes j} (dM) dX).
$$
We immediately deduce that
$w^N_{tj}(X)=\int_{R^j} m_1\dots m_j F^{N,M}_{tj}(X)r_0^{\otimes j}(dM)$ goes strongly in $L^1((\R^2)^j)$
to $w_t^{\otimes j}(X)= \int_{R^j} m_1\dots m_j (\prod_{i=1}^j f^{m_i}_{t}(x_i)) r_0^{\otimes j}(dM)$,
since $\AA$ is compact.
\end{preuve}


\begin{thebibliography}{99}

\bibitem{Aalto}
{ \sc Aalto, D. and Kinnunen, J.}
\newblock Maximal functions in Sobolev spaces.  
\newblock { \em Sobolev spaces in mathematics. I,} 25--67,
Int. Math. Ser. (N. Y.), 8, Springer, New York, 2009. 

\bibitem{BeRoVa}
{\sc Benachour, S., Roynette, B. and Vallois, P.}
\newblock Branching process associated with 2d-Navier Stokes equation.
\newblock {\em  Rev. Mat. Iberoamericana 17}, 2 (2001), 331--373.

\bibitem{BenArtzi}
{\sc Ben-Artzi, M.}
\newblock Global solutions of two-dimensional {N}avier-{S}tokes and {E}uler
  equations.
\newblock {\em Arch. Rational Mech. Anal. 128}, 4 (1994), 329--358.


\bibitem{BGG}
{\sc Bolley, F., Gentil, I. and Guillin, A.}
\newblock Uniform convergence to equilibrium for granular media.
\newblock To appear in {\em  Arch. Rational Mech. Anal.}
 
 

\bibitem{Bouchet}
{\sc Bouchet,F and Sommeria, J.}
\newblock Statistical mechanics for geophysical flows. From Jovian atmosphere
modelization to the parameterization of small scale turbulence.
\newblock  World Scientific, Hackensack, 2012.
  

\bibitem{BrezisBook}
{\sc Brezis, H.}
\newblock {\em Analyse fonctionnelle}.
\newblock Collection Math\'ematiques Appliqu\'ees pour la Ma\^\i trise.
  [Collection of Applied Mathematics for the Master's Degree]. Masson, Paris,
  1983.
\newblock Th{\'e}orie et applications. [Theory and applications].

\bibitem{Brezis}
{\sc Brezis, H.}
\newblock Remarks on the preceding paper by {M}. {B}en-{A}rtzi: ``{G}lobal
  solutions of two-dimensional {N}avier-{S}tokes and {E}uler equations''
  [{A}rch.\ {R}ational {M}ech.\ {A}nal.\ {\bf 128} (1994), no.\ 4, 329--358].
\newblock {\em Arch. Rational Mech. Anal. 128}, 4 (1994), 359--360.


\bibitem{BrezisConvex}
{\sc Brezis, H.}
\newblock Convergence in {${\DD}'$} and in {$L^1$} under strict convexity.
\newblock In {\em Boundary value problems for partial differential equations
  and applications}, vol.~29 of {\em RMA Res. Notes Appl. Math.} Masson, Paris,
  1993, 43--52.

\bibitem{CLMP1}
{\sc Caglioti, E., Lions, P.-L., Marchioro, C. and  Pulvirenti, M.}
\newblock A special class of stationary flows for two - dimensional Euler equation: 
a statistical mechanics description.
\newblock  {\em Comm. Math. Phys. 143}, (1992), 501--525. 

\bibitem{Carlen91}
{\sc Carlen, E.~A.}
\newblock Superadditivity of {F}isher's information and logarithmic {S}obolev
  inequalities.
\newblock {\em J. Funct. Anal. 101}, 1 (1991), 194--211.

\bibitem{CCLLV}
{\sc Carlen, E., Carvalho, M., Le Roux, J., Loss, M., and Villani, C.}
\newblock Entropy and Chaos in the Kac's model.
\newblock {\em Kinet. Relat. Models 3}, 1 (2010), 85--122. 

\bibitem{CepaLepingle}
{\sc C{\'e}pa, E. and L{\'e}pingle, D. }
\newblock Diffusing particles with electrostatic repulsion. 
\newblock {\em Probab. Theory Related Fields 107},  4 (1997), 429--449. 

\bibitem{CrippaDeLellis}
{\sc Crippa, G., and De~Lellis, C.}
\newblock Estimates and regularity results for the {D}i{P}erna-{L}ions flow.
\newblock {\em J. Reine Angew. Math. 616} (2008), 15--46.

\bibitem{deFinetti1937}
{\sc de~Finetti, B.}
\newblock La pr\'evision : ses lois logiques, ses sources subjectives.
\newblock {\em Ann. Inst. H. Poincar\'e 7}, 1 (1937), 1--68.

\bibitem{dPL-BFP}
{\sc DiPerna, R.~J., and Lions, P.-L.}
\newblock On the {F}okker-{P}lanck-{B}oltzmann equation.
\newblock {\em Comm. Math. Phys. 120}, 1 (1988), 1--23.


\bibitem{dPL}
{\sc DiPerna, R.~J., and Lions, P.-L.}
\newblock Ordinary differential equations, transport theory and {S}obolev spaces.
\newblock {\em Invent. Math. 98}, 3 (1989), 511--547.


\bibitem{Figalli}
{\sc Figalli, A.}
\newblock Existence and uniqueness of martingale solutions for {SDE}s with
  rough or degenerate coefficients.
\newblock {\em J. Funct. Anal. 254}, 1 (2008), 109--153.

\bibitem{FGP}
{\sc Flandoli, F., Gubinelli, M. and Priola, E.}
\newblock  Full well-posedness of point vortex dynamics corresponding to stochastic 2D Euler equations. 
\newblock {\em Stochastic Process. Appl. 121} (2011), 1445--1463.

\bibitem{FontMart}
{\sc Fontbona, J. and  Martinez, M.}
\newblock Paths clustering and an existence result for stochastic vortex systems. 
\newblock {\em J. Stat. Phys. 128}, 3 (2007), 699--719.

\bibitem{gg}
{\sc Gallagher, I., and Gallay, T.}
\newblock Uniqueness for the two-dimensional {N}avier-{S}tokes equation with a
  measure as initial vorticity.
\newblock {\em Math. Ann. 332}, 2 (2005), 287--327.

\bibitem{Gallay}
{\sc Gallay, Th.}
\newblock Interaction of vortices in weakly viscous planar flows.
\newblock {\em Archive Rat. Mech. Anal. 200} (2011), 445--490.

\bibitem{GallayWayne}
{\sc Gallay, Th. and  Wayne, C.E.}
\newblock Global stability of vortex solutions of the two-dimensional
Navier-Stokes equation.
\newblock {\em Commun. Math. Phys. 255,} (2005), 97--129.

\bibitem{HaurayJabin}
{\sc Hauray, M. and Jabin, P.-E.}
\newblock Propagation of chaos for particles approximations of Vlasov
equations with singular forces.
\newblock http://hal.archives-ouvertes.fr/hal-00609453.

\bibitem{hm}
{\sc Hauray, M. and Mischler, S.}
\newblock On {K}ac's chaos and related problems.
\newblock http://hal.archives-ouvertes.fr/hal-00682782.

\bibitem{Helmholtz}
{\sc Helmholtz, H.}, 
\newblock \"Uber Integrale der hydrodynamischen Gleichungen, welche den Wirbelbewegungen ent-sprechen.
\newblock {\em J. Reine Angew. Math. 55} (1858), 25--55.

\bibitem{Hewitt-Savage}
{\sc Hewitt, E., and Savage, L.~J.}
\newblock Symmetric measures on {C}artesian products.
\newblock {\em Trans. Amer. Math. Soc. 80} (1955), 470--501.

\bibitem{Kac1956}
{\sc Kac, M.}
\newblock Foundations of kinetic theory.
\newblock In {\em Proceedings of the {T}hird {B}erkeley {S}ymposium on
  {M}athematical {S}tatistics and {P}robability, 1954--1955, vol. {III}}
  (Berkeley and Los Angeles, 1956), University of California Press,
  171--197.

\bibitem{Kirchhoff}
{\sc Kirchhoff, G.} 
\newblock Vorlesungen \"Uber Math. Phys. 
\newblock {\em Teuber}, Leipzig (1876). 

\bibitem{LeBrisLions}
{\sc Le~Bris, C., and Lions, P.-L.}
\newblock Existence and uniqueness of solutions to {F}okker-{P}lanck type
  equations with irregular coefficients.
\newblock {\em Comm. Partial Differential Equations 33}, 7-9 (2008),
  1272--1317.

\bibitem{LL}
{\sc Lieb, E. and Loss, M.}
\newblock Analysis.
\newblock {\em Graduate Studies in Mathematics, 14},   AMS, Providence, 1997.


\bibitem{Malrieu}
{\sc  Malrieu, F.}
\newblock  Convergence to equilibrium for granular media equations and their Euler schemes. 
\newblock {\em Ann. Appl. Probab.13,} 2 (2003), 540--560.
 

\bibitem{Marchioro}
{\sc  Marchioro, C.}
\newblock On the inviscid limit for a fluid with a concentrated vorticity. 
\newblock {\em Comm. Math. Phys. 196} (1998), 53--65.

\bibitem{MarPul2}
{\sc Marchioro, C. and  Pulvirenti, M.} 
\newblock Hydrodynamics in two dimensions and vortex theory.
\newblock {\em Comm. Math. Phys. 84}, 4 (1982), 483--503. 

\bibitem{MarPul}
{\sc Marchioro, C. and  Pulvirenti, M.} 
\newblock Mathematical theory of incompressible nonviscous fluids. 
\newblock {\em Applied Mathematical Sciences, 96}, Springer-Verlag, New York, 1994.

\bibitem{McK3}
{\sc McKean, Jr., H.~P.}
\newblock Propagation of chaos for a class of non-linear parabolic equations.
\newblock In {\em Stochastic {D}ifferential {E}quations ({L}ecture {S}eries in
  {D}ifferential {E}quations, {S}ession 7, {C}atholic {U}niv., 1967)}. Air
  Force Office Sci. Res., Arlington, Va., 1967, 41--57.
  
\bibitem{m}{\sc M{\'e}l{\'e}ard, S.}
\newblock Asymptotic behaviour of some interacting particle systems;
  {M}c{K}ean-{V}lasov and {B}oltzmann models.
\newblock In {\em Probabilistic models for nonlinear partial differential
  equations ({M}ontecatini {T}erme, 1995)}, vol.~1627 of {\em Lecture Notes in
  Math.} Springer, Berlin, 1996, 42--95.

\bibitem{m2}
{\sc M{\'e}l{\'e}ard, S.}
\newblock Monte-Carlo approximations for 2d Navier-Stokes equations with measure initial data.
\newblock {\em Probab. Theory Related Fields} 121 (2001), 367--388.

\bibitem{MMKacProg}
{\sc Mischler, S. and Mouhot, C.}
\newblock Kac's Program in Kinetic Theory. 
\newblock To appear in {\em  Inventiones Mathematicae} 

\bibitem{MMW}
{\sc Mischler, S., Mouhot, C. and Wennberg, B.}
\newblock A new approach to quantitative chaos propagation for drift, diffusion and jump processes. 
\newblock  arXiv:1101.4727


\bibitem{Onsager}
{\sc Onsager, L.}
\newblock Statistical hydrodynamics.
\newblock  {\em Il Nuovo Cimento 6} (1949), 279--287.

\bibitem{oDiff}
{\sc Osada, H.}
\newblock Diffusion processes with generators of generalized divergence form.
\newblock  {\em J. Math. Kyoto Univ. 27},  4 (1987), 597--619. 

\bibitem{o2}
{\sc Osada, H.}
\newblock A stochastic differential equation arising from the vortex problem.
\newblock {\em Proc. Japan Acad. Ser. A Math. Sci. 61}, 10 (1986), 333--336.

\bibitem{o1}
{\sc Osada, H.}
\newblock Propagation of chaos for the two-dimensional {N}avier-{S}tokes
  equation.
\newblock {\em Proc. Japan Acad. Ser. A Math. Sci. 62}, 1 (1986), 8--11.

\bibitem{o0}
{\sc  Osada, H.}
\newblock Propagation of chaos for the two-dimensional Navier-Stokes equation. 
\newblock { \em Probabilistic methods in mathematical physics}  (Katata/Kyoto, 1985), 303--334, 
Academic Press, Boston, MA, 1987.

\bibitem{Osada}
{\sc Osada, H.}
\newblock Limit points of empirical distributions of vortices with small
  viscosity.
\newblock In {\em Hydrodynamic behavior and interacting particle systems
  ({M}inneapolis, {M}inn., 1986)}, vol.~9 of {\em IMA Vol. Math. Appl.}
  Springer, New York, 1987, 117--126.


\bibitem{RR}
{\sc Robinson, D.~W., and Ruelle, D.}
\newblock Mean entropy of states in classical statistical mechanics.
\newblock {\em Comm. Math. Phys. 5} (1967), 288--300.

\bibitem{Schochet}
{\sc Schochet, S.}
\newblock  The point vortex method for periodic weak solutions of the 2-D Euler equations.
\newblock  {\em Comm. Pure Appl. Math. 49} (1996) 911--965.

\bibitem{Stein}
{\sc Stein, E.~M.}
\newblock {\em Harmonic Analysis}.
\newblock Princeton Mathematical Series, No. 43. Princeton University Press,
  Princeton, N.J., 1993.

\bibitem{S1}
{\sc Sznitman, A.-S.}
\newblock \'{E}quations de type de {B}oltzmann, spatialement homog\`enes.
\newblock {\em Z. Wahrsch. Verw. Gebiete 66}, 4 (1984), 559--592.


\bibitem{SSF}
{\sc Sznitman, A.-S.}
\newblock Topics in propagation of chaos.
\newblock In {\em \'{E}cole d'\'{E}t\'e de {P}robabilit\'es de {S}aint-{F}lour
  {XIX}---1989}, vol.~1464 of {\em Lecture Notes in Math.} Springer, Berlin,
  1991, 165--251.

\bibitem{Takanobu}
{\sc Takanobu, S.} 
\newblock On the existence and uniqueness of SDE describing an n-particle system interacting
via a singular potential. 
\newblock {\em Proc. Japan Acad. Ser. A Math. Sci. 6} (1985) 287--290.

\bibitem{yw}
{\sc Yamada, T., and Watanabe, S.}
\newblock On the uniqueness of solutions of stochastic differential equations.
\newblock {\em J. Math. Kyoto Univ. 11} (1971), 155--167.



\end{thebibliography}
\end{document}